\documentclass[12pt]{amsart}

\usepackage{etex, amsfonts, amsmath, amsthm, amssymb, stmaryrd, epsfig, graphics,
  psfrag, latexsym, color, mathtools, mathrsfs,enumitem, subfig,
  longtable, booktabs}
\usepackage[all,cmtip]{xy}
\usepackage[percent]{overpic}
\definecolor{darkblue}{rgb}{0,0,0.4}
\usepackage{xr-hyper}
\usepackage[colorlinks=true,
citecolor=darkblue, filecolor=darkblue, linkcolor=darkblue,
urlcolor=darkblue]{hyperref} \usepackage[all]{hypcap}

\graphicspath{{draws/}{}}

\numberwithin{equation}{section}
\newtheorem{lem}{Lemma}[section]

\newtheorem{theorem}[lem]{Theorem}
\newtheorem{lemma}[lem]{Lemma}

\newtheorem{corollary}[lem]{Corollary}

\newtheorem{proposition}[lem]{Proposition}

\theoremstyle{definition}

\newtheorem{definition}[lem]{Definition}

\theoremstyle{remark}

\newtheorem{remark}[lem]{Remark}

\newtheorem{idea}{Idea}
\numberwithin{figure}{section}

\newcommand{\Appendix}[1]{\hyperref[app:#1]{Appendix~\ref*{app:#1}}}
\newcommand{\Section}[1]{\hyperref[sec:#1]{Section~\ref*{sec:#1}}}
\newcommand{\Subsection}[1]{\hyperref[subsec:#1]{Subsection~\ref*{subsec:#1}}}
\newcommand{\Lemma}[1]{\hyperref[lem:#1]{Lemma~\ref*{lem:#1}}}
\newcommand{\Theorem}[1]{\hyperref[thm:#1]{Theorem~\ref*{thm:#1}}}
\newcommand{\Citethm}[1]{\hyperref[citethm:#1]{Theorem~\ref*{citethm:#1}}}
\newcommand{\Definition}[1]{\hyperref[def:#1]{Definition~\ref*{def:#1}}}
\newcommand{\Remark}[1]{\hyperref[rem:#1]{Remark~\ref*{rem:#1}}}
\newcommand{\Figure}[1]{\hyperref[fig:#1]{Figure~\ref*{fig:#1}}}
\newcommand{\Conjecture}[1]{\hyperref[conj:#1]{Conjecture~\ref*{conj:#1}}}
\newcommand{\Corollary}[1]{\hyperref[cor:#1]{Corollary~\ref*{cor:#1}}}
\newcommand{\Proposition}[1]{\hyperref[prop:#1]{Proposition~\ref*{prop:#1}}}
\newcommand{\Question}[1]{\hyperref[ques:#1]{Question~\ref*{ques:#1}}}
\newcommand{\Example}[1]{\hyperref[exam:#1]{Example~\ref*{exam:#1}}}
\newcommand{\Table}[1]{\hyperref[table:#1]{Table~\ref*{table:#1}}}
\newcommand{\Restric}[1]{\hyperref[restric:#1]{Restriction~\ref*{restric:#1}}}

\newcommand{\R}{\mathbb{R}}
\newcommand{\Z}{\mathbb{Z}}

\newcommand{\mf}{\mathfrak}

\newcommand{\wt}{\widetilde}
\newcommand{\ol}[1]{m(#1)}

\newcommand{\si}{\sigma}
\newcommand{\al}{\alpha}
\newcommand{\be}{\beta}

\newcommand{\ep}{\epsilon}

\newcommand{\from}{\colon}

\newcommand{\onto}{\twoheadrightarrow}

\newcommand{\set}[2]{\{#1\mid#2\}}

\newcommand{\HFK}{\mathit{HFK}}

\DeclareMathOperator{\Id}{Id}
\DeclareMathOperator{\Sq}{Sq}

\DeclareMathOperator{\Hom}{Hom}

\newcommand{\Kh}{\mathit{Kh}}

\newcommand{\KhCx}{\mathcal{C}_{\Kh}}
\newcommand{\Cat}{\mathscr{C}}

\newcommand{\Realize}[2][{}]{|#2|_{#1}}

\newcommand{\CubeFlowCat}{\mathscr{C}_C}
\newcommand{\KhFlowCat}{\mathscr{C}_{\mathit{Kh}}}

\newcommand{\gr}{\mathrm{gr}}

\newcommand{\KhSpace}{\mathcal{X}_\mathit{Kh}}

\newcommand{\TupV}{\mathbf}
\newcommand{\Frame}{\varphi}
\newcommand{\Cube}{\mathcal{C}}

\newcommand{\diff}{\delta}
\newcommand{\diffKh}{\diff_{\Kh}}
\newcommand{\diffLee}{\diff_{\mathit{Lee}}}

\newcommand{\co}{\colon}
\newcommand{\bdy}{\partial}

\newcommand{\DD}{\mathbb{D}}

\newcommand{\ZZ}{\mathbb{Z}}

\DeclareMathOperator{\image}{im}
\DeclareMathOperator{\Span}{span}

\newcommand{\BNcx}{\mathcal{C}}
\newcommand{\LeeCx}{\mathcal{C}_{\mathit{Lee}}}
\newcommand{\Filt}{\mathcal{F}}
\newcommand{\Red}{\textcolor{red}}
\newcommand{\xbar}{x_{\text{\rotatebox[origin=c]{90}{$-$}}}}
\newcommand{\Field}{\mathbb{F}}

\newcommand{\SL}{\mathit{sl}}
\newcommand{\maxsl}{\overline{\mathit{sl}}}
\newcommand{\TB}{\mathit{tb}}
\newcommand{\rot}{\mathit{rot}}
\newcommand{\lk}{\operatorname{lk}}

\newcommand{\std}{\mathit{std}}
\DeclareMathOperator{\Cone}{Cone}
\newcommand{\st}{^\text{st}}
\newcommand{\sphere}{\mathbb{S}}
\newcommand{\SZ}{\mathit{SZ}}
\newcommand{\ppsi}{\psi^{+}}
\newcommand{\npsi}{\psi^{-}}
\newcommand{\pnpsi}{\psi}
\newcommand{\dpsi}{\psi^{\mathit{diff}}}
\begin{document}

\title{On transverse invariants from Khovanov homology}

\author{Robert Lipshitz}
\thanks{RL was partly supported by NSF grant DMS-1149800 and a Sloan Research Fellowship.}
\email{\href{mailto:lipshitz@math.columbia.edu}{lipshitz@math.columbia.edu}}
\address{Department of Mathematics, University of North Carolina,
  Chapel Hill, NC 27599}

\author{Lenhard Ng}
\thanks{LN was partly supported by NSF grant DMS-0846346.}
\email{\href{mailto:ng@math.duke.edu}{ng@math.duke.edu}}
\address{Department of Mathematics, Duke University,
  Durham, NC 27708}

\author{Sucharit Sarkar}
\thanks{SS was supported by a Clay Mathematics Institute Postdoctoral Fellowship}
\email{\href{mailto:sucharit@math.princeton.edu}{sucharit@math.princeton.edu}}
\address{Department of Mathematics, Princeton University, Princeton,
  NJ 08544}

\subjclass[2010]{\href{http://www.ams.org/mathscinet/search/mscdoc.html?code=57M25,57R17}{57M25,57R17}}

\keywords{}

\date{\today}

\begin{abstract}
  In~\cite{Plam-06-KhTrans}, O.~Plamenevskaya associated to each
  transverse knot $K$ an element of the Khovanov homology of $K$. In
  this paper, we give two refinements of Plamenevskaya's invariant,
  one valued in Bar-Natan's deformation
  (from~\cite{Bar-kh-tangle-cob}) of the Khovanov complex and another
  as a cohomotopy element of the Khovanov spectrum
  (from~\cite{RS-khovanov}). We show that the first of these
  refinements is invariant under negative flypes and $\SZ$ moves; this
  implies that Plamenevskaya's class is also invariant under these
  moves. We go on to show that for small-crossing transverse knots
  $K$, both refinements are determined by the classical invariants of
  $K$.
\end{abstract}

\maketitle


\section{Introduction}\label{sec:intro}

Transverse links have emerged as central objects of study in
three-dimensional contact geometry. We will restrict our attention to
transverse links in standard contact $\R^3$: by a \textit{transverse
  link}, we mean a knot or link in $\R^3$ that
is everywhere transverse to the $2$-plane field $\ker(dz-y\,dx)$. One
typically studies transverse links up to \textit{transverse isotopy},
or isotopy through a one-parameter family of transverse links.

Up to transverse isotopy, transverse links have two \emph{classical
invariants}: the underlying topological link type and the self-linking
number $\SL \in \Z$. (Strictly speaking, for a multi-component link,
each component has a self-linking number; see Remark~\ref{rmk:sl}.)
A natural question is whether there are
transverse links that have the same classical invariants but are not
transversely isotopic. The answer is yes, with the earliest examples
given by Etnyre--Honda \cite{EtnyreHonda-cabling} and Birman--Menasco \cite{BirmanMenasco-transverse}, but the question is
surprisingly subtle. By contrast, the corresponding question for
Legendrian links, which are also central to contact geometry, was
answered considerably earlier by Chekanov \cite{Chekanov-DGA}.

One approach to distinguishing transverse links with the same
classical invariants is to introduce further, more refined invariants
of transverse links. The first candidate for such an invariant was
introduced by Plamenevskaya \cite{Plam-06-KhTrans}: to a transverse
link of topological type $K$, this associates a distinguished class in
the Khovanov homology of $K$. Since Plamenevskaya's groundbreaking
work, transverse invariants of a similar flavor have been discovered
by Wu \cite{Wu-transverse} in $\mathfrak{sl}_n$ Khovanov--Rozansky
homology, and by Ozsv\'ath--Szab\'o--Thurston
\cite{OSzT-hf-legendrian} and Lisca--Ozsv\'ath--Stipsicz--Szab\'o
\cite{LOSS} in knot Floer homology. (Knot contact homology produces a
transverse invariant of a somewhat different flavor, called transverse
homology \cite{EENS-transverse,Ng-transhom}.)

\emph{A priori}, it might be the case that some of these invariants are
determined by the smooth link type and self-linking number.
An invariant of transverse links is called
\textit{effective} if it achieves different values for some pair of
transverse links with the same classical invariants. It is known that
the invariant in knot Floer homology is effective
\cite{NOT-hf-transverse}, as is transverse homology
\cite{Ng-transhom}.

It has long been an open question whether
Plamenevskaya's original invariant in Khovanov homology is effective.
One goal of this paper is to shed some light on this question, although
we do not resolve it. To produce candidates for distinct
transverse links with the same classical invariants, two techniques
are commonly considered: negative braid flypes on braids and $\SZ$ moves on
Legendrian links. (There are other techniques as well, but these two
seem to be the most successful for small-crossing knots.) In fact,
we will show that negative braid flypes and $\SZ$ moves are equivalent for transverse
knots; see Proposition~\ref{prop:flype-SZ}.

In Theorem~\ref{thm:flype-inv}, we show that the Plamenevskaya
invariant is unchanged by negative braid flypes, and thus by $\SZ$
moves as well.  Our result can be seen as evidence that the
Plamenevskaya invariant may not be effective; by contrast, both the
$\HFK$ transverse invariant and transverse homology can distinguish
transverse knots related by these moves. There are transverse links
that are known not to be related by these moves, but the simplest one
known to the authors is topologically a certain cable of the trefoil
knot \cite{EtnyreLafountainTosun} and (we believe) outside the reach
of current technology for computing the Plamenevskaya invariant.

In a different direction, we give three refinements of the
Plamenevskaya invariant, which could be effective even if the original
invariant is not. The first of these is a filtered version of the
Plamenevskaya invariant, living in Bar-Natan's (or, if one prefers,
Lee's) deformation of the Khovanov
complex~\cite{Bar-kh-tangle-cob,Lee-kh-endomorphism}; see
Theorem~\ref{thm:filt-inv}. In fact, this filtered invariant comes in
two versions, and subtracting the two versions gives another
transverse invariant, \emph{a priori} incomparable to the filtered or
original Plamenevskaya invariants; see
Theorem~\ref{thm:diff-is-invt}. Finally, a space-level version
$\KhSpace(K)$ of Khovanov homology was recently
constructed~\cite{RS-khovanov}, and the (original) Plamenevskaya invariant admits
a refinement as an element of the stable cohomotopy (rather than
cohomology) groups of $\KhSpace(K)$; see
Theorem~\ref{thm:cohtpy}. This leads to a number of computable
auxiliary invariants; see Section~\ref{sec:computable}.

Unfortunately, we have also been unable to show that any of these
refinements is effective. In particular, the filtered invariants are
unchanged by negative flypes and $\SZ$ moves.  There are also simpler
structural results that mean that for small-crossing knots, the
refinements have no non-classical information; see
Sections~\ref{sec:sad} and~\ref{sec:computable}.

However, there are some indications that some of the refinements may stand
a better chance of being effective than the original Plamenevskaya
invariant by itself.
For instance, over a ring where $2$ is
invertible, the original invariant agrees for any two transverse knots
that become the same after one stabilization (see
Proposition~\ref{prop:olgastabinv} for the exact result), while we do
not know if this is true for the filtered invariant. We know even less
about the behavior of the cohomotopy invariant; it might even be able
to distinguish transverse knots related by negative
flypes.

We remain
optimistic that, with more work on computational tools, for more
complicated knots, both the filtered
Plamenevskaya invariant and the cohomotopy Plamenevskaya invariant
will turn out to be effective. For the latter, we include some
discussion in Section~\ref{sec:computable} of some possible ways that
effectiveness might be tested.

\subsection*{Acknowledgments} We thank Olga Plamenevskaya and Jacob
Rasmussen for useful conversations. We also thank the referees for
their helpful suggestions.

\section{Some constructions in contact geometry}\label{sec:contact}

Here we review some well-known facts involving transverse knots and
links, referring the reader to the survey paper \cite{Etnyresurvey}
for general background on transverse knots and links, and to
\cite{NgThurston,KhandhawitNg} for further background on
some of the specifics that we give
here. By a \emph{transverse link} (respectively, \emph{Legendrian
  link}) we mean a knot or a link in $\R^3$ that is everywhere
transverse (respectively, tangent) to the standard contact structure
$\ker(dz-y\,dx)$; by a transverse knot (respectively, \emph{Legendrian
  knot}) we mean a transverse link (respectively, Lendendrian link) of
one component. There are two well-known
correspondences, one between transverse links and equivalence classes
of braids, and another
between transverse links and equivalence classes of Legendrian
links. We describe each of
these in turn, and then present a result linking them.

\subsection{Transverse links and braids}\label{ssec:flype}

By work of Bennequin \cite{Bennequin}, any braid can be naturally
viewed as a transverse link (whose topological link type is the closure of the
braid), and every transverse link arises in this
way (up to transverse isotopy) from some braid. The transverse Markov
Theorem \cite{OS-knot-transMarkov,Wri-knot-transMarkov} states that
under this correspondence, transverse links up to transverse isotopy
can be identified with braids up to
conjugation and positive braid
stabilization and destabilization,
\[
B \longleftrightarrow B\sigma_m,
\]
where $B$ is an element of the braid group $B_m$ and $B\sigma_m \in
B_{m+1}$. We refer to conjugation and positive (de)stabilization as
\textit{transverse Markov moves}.
The self-linking number of a transverse link $T$ can be
expressed in terms of a corresponding braid $B$ as $\SL(T) =
w-m$, where $w$ is the writhe of $B$ (the sum of the
exponents of the braid word) and $m$ is the braid index of $B$.

\begin{remark}\label{rmk:sl}
For a transverse link $T = T_1 \cup \cdots \cup T_r$ with $r \geq 2$
components, the self-linking number of $T$ is related to the
self-linking number of its components as follows:
\[
\SL(T) = \sum_i \SL(T_i) + 2 \sum_{i<j} \lk(T_i,T_j),
\]
where $\lk$ represents the (topological) linking number. This follows
from the formula $\SL(T)=w-m$ above, or from any of the standard
definitions of self-linking number in contact geometry (cf.\
Section~\ref{ssec:SZ}). The $r$-tuple
$(\SL(T_1),\ldots,\SL(T_r)) \in \Z^r$ is invariant under transverse
isotopy, and in some sense should be considered the true ``self-linking
number'' associated to $T$. However, for consistency we will refer to
the single integer $\SL(T)$ as the self-linking number.
\end{remark}

There is also a notion of negative braid
stabilization and destabilization,
\[
B \longleftrightarrow B\sigma_m^{-1},
\]
for $B\in B_m$. This descends to a well-defined operation on
transverse links called \emph{transverse stabilization}, which decreases
self-linking number by $2$. Given any two transverse links
representing the same topological link type, one can perform some
number of transverse stabilizations to each to obtain transversely
isotopic links.

Birman and Menasco
introduced a class of flype operations on
certain braids that preserves the braid index as well as the
topological link type of the braid closure; see
\cite{BirmanMenasco-3braid} for $3$-braids,
generalized in \cite{BirmanMenasco-MTWSI} for arbitrary braid
index. In this paper, we will use ``flype'' to mean the following.

\begin{definition}\label{def:flype}
   Let $A, B\in B_m$ be braids and $k\in\ZZ$. We say that the braids
\[
A \sigma_m^k B \sigma_m,
\hspace{3ex}
A \sigma_m B \sigma_m^k,
\]
which are elements in $B_{m+1}$,
are related by a \emph{positive flype}. Similarly, we say that
\[
A\sigma_m^{k}B\sigma_m^{-1},
\hspace{3ex}
A\sigma_m^{-1}B\sigma_m^{k}
\]
   are related by a \emph{negative flype}. See
   Figure~\ref{fig:flype}.
\end{definition}

\begin{figure}
  \centering
  \includegraphics[width=4in]{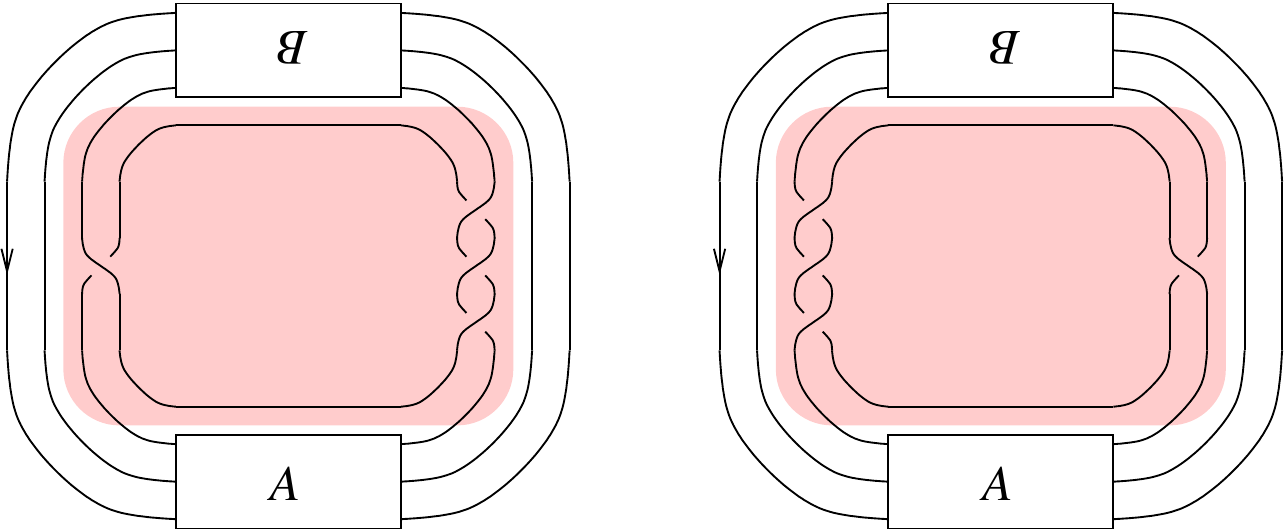}
  \caption{\textbf{Negative flype.}
    Pictured here: the closures of braids $A \sigma_3^3 B \sigma_3
    \sigma_3^{-1}$ (left) and $A \sigma_3^{-1} B \sigma_3^3$ (right),
    for $A,B \in B_3$. Note that the diagrams are related by a
    standard topological flype in the shaded regions.
}
  \label{fig:flype}
\end{figure}

\noindent
(Note that Birman and Menasco's original definition of flypes in
\cite{BirmanMenasco-MTWSI} is actually a more general, weighted
version of the flypes that we consider here.)
In addition to link type and braid index, both flypes preserve the
self-linking number of the corresponding transverse link.

One can express a negative flype as a
composition of conjugation, one negative braid stabilization, and one
negative braid destabilization. Since this will be important to us, we
write down a precise sequence of braid moves here, labeled by the
relevant braid operations as well as the corresponding Reidemeister
moves on the diagram for the closed braid (cf.\
\cite[Figure~5]{BirmanMenasco-MTWSI}):
\begin{equation}
  \label{eq:flype-moves}
\begin{aligned}
 A\sigma_m^{k}B\sigma_m^{-1}
&\to A \sigma_m^{-1}\sigma_m^{k+1} B \sigma_m^{-1} &&
\text{(Reidemeister II)}
\\
&\to A \sigma_m^{-1}\sigma_{m+1}^{-1}\sigma_{m}^{k+1} B \sigma_m^{-1}
&& \text{(negative braid stabilization; Reidemeister I)}\\
&\to A \sigma_{m+1}^{k+1}\sigma_m^{-1}\sigma_{m+1}^{-1} B
\sigma_m^{-1} &&
\text{(braid relation; Reidemeister III)}\\
&\to \sigma_{m+1}^{k+1} A \sigma_m^{-1}B \sigma_{m+1}^{-1}
\sigma_m^{-1} &&
\text{(braid relation---far commutativity)}\\
&\to A \sigma_m^{-1}B \sigma_{m+1}^{-1}
\sigma_m^{-1} \sigma_{m+1}^{k+1} &&
\text{(braid conjugation)}\\
&\to A \sigma_m^{-1}B \sigma_m^{k+1} \sigma_{m+1}^{-1}
\sigma_m^{-1} &&
\text{(braid relation; Reidemeister III)}\\
&\to A \sigma_m^{-1}B \sigma_m^{k+1}
\sigma_m^{-1} &&
\text{(negative braid destabilization; Reidemeister I)}\\
&\to A \sigma_m^{-1} B \sigma_m^k &&
\text{(Reidemeister II)}.
\end{aligned}
\end{equation}
Since these moves involve negative braid (de)stabilization, they do
not yield a transverse isotopy, and indeed there are many
examples of negative flypes producing distinct transverse links; see
e.g.~\cite{BirmanMenasco-transverse,KhandhawitNg}.

In a completely analogous way, one can express a positive flype as a
composition of braid conjugation, one positive braid stabilization,
and one positive braid destabilization. In this case it follows from
the transverse Markov theorem that positive flypes preserve transverse
type.

\begin{definition}\label{def:flype-eq}
Two transverse links $T,T'$ are \emph{flype-equivalent} if there are transverse links
$T_0=T,T_1,\ldots,T_k=T'$ such that for each $i=1,\ldots,k$, $T_{i-1}$ and $T_i$ can be represented by braids that differ by a flype.
\end{definition}
\noindent Equivalently, transverse links are flype-equivalent if they have braid representatives that are related by a sequence of
the following braid moves: braid conjugation, positive braid (de)stabilization, and negative flypes. (Note in particular that since positive flypes
can be expressed in terms of conjugation and positive
(de)stabilization, they can be omitted here and in
Definition~\ref{def:flype-eq} if desired.)

If $B_1,B_2$ are braids related by a negative flype and $T_1,T_2$ are the corresponding transverse links, then by \eqref{eq:flype-moves}, $S(T_1)$ and $S(T_2)$ are transversely isotopic, where $S$ represents transverse stabilization. Thus we have the following result.

\begin{proposition}\label{prop:transverse-equivalence-list}
Let $T,T'$ be transverse links. Then the following properties satisfy $(\ref{flype1}) \Rightarrow (\ref{flype2}) \Rightarrow (\ref{flype3})$:
\begin{enumerate}
\item\label{flype1}
$T,T'$ are flype-equivalent;
\item\label{flype2}
$S(T)$ and $S(T')$ are transversely isotopic;
\item\label{flype3}
$\SL(T) = \SL(T')$ and $T,T'$ are topologically isotopic.
\end{enumerate}
\end{proposition}

It is known that (\ref{flype3}) does not necessarily imply (\ref{flype2}): by \cite{EtnyreLafountainTosun}, there are transverse knots representing certain
cables of torus knots that have the same self-linking number but
require an arbitrarily large number of stabilizations to become
transversely isotopic. We do not know if (\ref{flype2}) necessarily implies
(\ref{flype1}), although it seems unlikely.

\subsection{Transverse and Legendrian links}\label{ssec:SZ}

Another approach to transverse links is through Legendrian links. Any
(oriented) Legendrian link can be $C^0$ perturbed to a well-defined
transverse link, its \emph{(positive) transverse pushoff}, and
Legendrian links that are Legendrian isotopic have transverse pushoffs
that are transversely isotopic. Conversely, any transverse link can be
$C^0$ perturbed to a Legendrian link called a \emph{Legendrian
  approximation}, though the Legendrian approximation is only
well-defined up to negative Legendrian stabilizations. (Recall that
there are two stabilization operations on Legendrian links, positive
and negative Legendrian stabilization $L \mapsto S_\pm(L)$, that are
well-defined on
Legendrian isotopy classes: in the front projection to the $xz$ plane,
each stabilization replaces a piece of the Legendrian link by a
two-cusped zigzag, where the cusps are oriented downwards
(respectively, upwards) for positive (respectively, negative) stabilization.)

Thus there is a many-to-one correspondence between Legendrian links
and transverse links, under which transverse links up to transverse
isotopy correspond to Legendrian links up to Legendrian isotopy and
negative Legendrian stabilization and destabilization. If $L$ is a
Legendrian link and $T$ is its transverse pushoff, then the
classical invariants of $L$ and $T$, the Thurston--Bennequin and
rotation numbers $\TB(L),\rot(L)$ and the self-linking number $\SL(T)$,
are related by $\SL(T) = \TB(L) - \rot(L)$. Furthermore, the
transverse stabilization
$S(T)$ is the transverse pushoff of $S_+(L)$.

\begin{figure}
  \centering
  \begin{overpic}[tics=10]{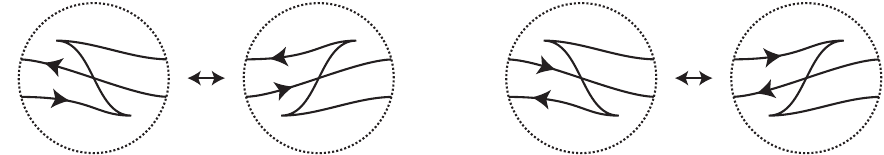}
    \put(21,-1) {(a)}
    \put(77,-1) {(b)}
  \end{overpic}
  \caption{An $\SZ_-$ move (left) and an $\SZ_+$ move (right) on
    Legendrian links.}
  \label{fig:sz}
\end{figure}

\begin{definition}\label{def:sz-eq}
Two Legendrian links $L,L'$ are \emph{$\SZ$-equivalent} if there are
Legendrian links $L_0=L,L_1,\ldots,L_k=L'$ such that for each
$i=1,\ldots,k$, $L_{i-1}$ and $L_i$ are Legendrian isotopic to
Legendrian links whose $xz$ projections are identical except for one
of the tangle replacements shown in Figure~\ref{fig:sz}. Similarly
define \emph{$\SZ_-$-equivalent} (respectively, \emph{$\SZ_+$-equivalent}) if we
restrict to only the tangle replacement on the left (respectively, right) of
Figure~\ref{fig:sz}.

Two transverse links are \emph{$\SZ$-equivalent} if they are the
transverse pushoffs of
$\SZ$-equivalent Legendrian links.
\end{definition}

Note that $\SZ$-equivalent Legendrian links are topologically isotopic and have the same Thurston--Bennequin and rotation numbers. (One can generalize the move shown in Figure~\ref{fig:sz} to arbitrary orientations, in which case $\TB$ is preserved but $\rot$ is not necessarily preserved.)
Indeed, the tangle replacements in Figure~\ref{fig:sz}, which
have previously appeared in the literature (e.g., \cite{EtnyreNg-problems})
though without the name ``$\SZ$ moves'',
are a key tool
in constructing Legendrian links with the same topological type and
classical invariants that are not necessarily Legendrian isotopic.
The Chekanov $5_2$ knots \cite{Chekanov-DGA} are $\SZ$-equivalent;
$\SZ$ moves also appear in some guise in
\cite{EpsteinFuchsMeyer,NOT-hf-transverse,EtnyreNgVertesi}, among
other papers.

Accordingly, $\SZ$-equivalent transverse links are topologically
isotopic and have the same self-linking number, and so they provide
good candidates for possibly distinct transverse links that share the same
classical invariants.
An observation of \cite{EpsteinFuchsMeyer}, in our language, states
that if $L,L'$ are $\SZ_+$-equivalent (respectively,
$\SZ_-$-equivalent), then $S_+(L)$ and $S_+(L')$ (respectively,
$S_-(L)$ and $S_-(L')$) are Legendrian isotopic. Thus we
have the following result.

\begin{proposition}
Let $L,L'$ be Legendrian links, with transverse pushoffs $T,T'$. If $L,L'$ are $\SZ_-$-equivalent, then $T,T'$ are transversely isotopic. Furthermore, the following properties satisfy $(\ref{sz1}) \Rightarrow (\ref{sz2}) \Rightarrow (\ref{sz3}) \Rightarrow (\ref{sz4})$:
\begin{enumerate}
\item\label{sz1}
$L,L'$ are $\SZ$-equivalent;
\item\label{sz2}
$T,T'$ are $\SZ$-equivalent;
\item\label{sz3}
$S(T),S(T')$ are transversely isotopic (recall $S$ denotes transverse stabilization);
\item\label{sz4}
$\SL(T) = \SL(T')$ and $T,T'$ are topologically isotopic.
\end{enumerate}
\end{proposition}

\begin{remark}
  Note that a \textit{negative} stabilization of a braid results in a
  transverse stabilization of the corresponding transverse link, while
  a \textit{positive} stabilization of a Legendrian link produces
  a transverse stabilization for the transverse pushoff.
\end{remark}

\subsection{Relation between equivalences}

We have discussed two techniques to produce candidates for transverse
links that are topologically equivalent and have the same self-linking
number but may not be transversely isotopic: negative flypes for
braids (Section~\ref{ssec:flype}) and $\SZ_+$ moves for Legendrian
links (Section~\ref{ssec:SZ}). Here we show that these two techniques
are identical.

\begin{proposition}\label{prop:flype-SZ}
Transverse links are flype-equivalent if and only if they are $\SZ$-equivalent.
\end{proposition}

Note that since positive flypes and $\SZ_-$ moves do not change
transverse type, one can replace ``$\SZ$-equivalent'' by ``$\SZ_+$-equivalent'', and similarly restrict to negative flypes, in
the statement of Proposition~\ref{prop:flype-SZ}.

\begin{figure}
  \centering
  \includegraphics[height=1.5in,width=5in]{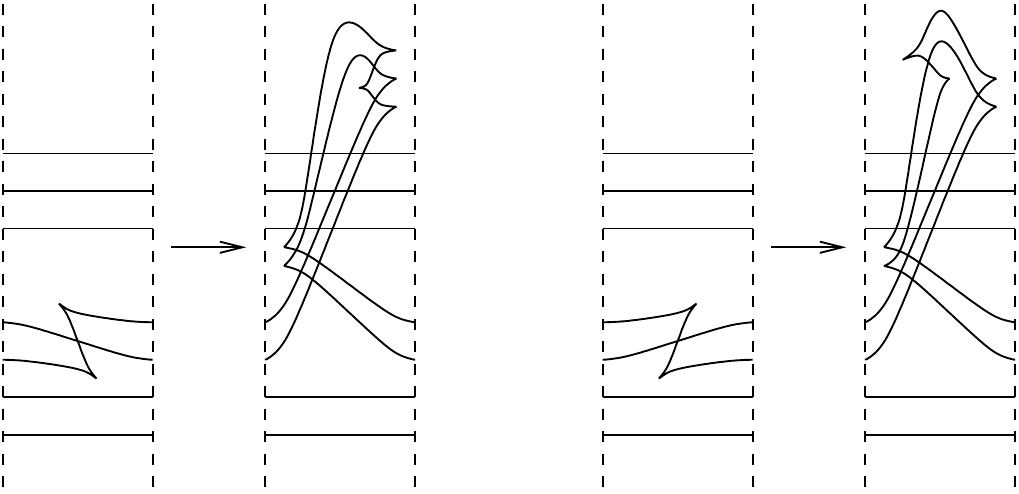}
  \caption{\textbf{Changing two fronts related by an $\SZ_+$ move.}
    Each is changed by a Legendrian isotopy.}
  \label{fig:SZdip}
\end{figure}

\begin{figure}
  \centering
  \includegraphics[height=1in]{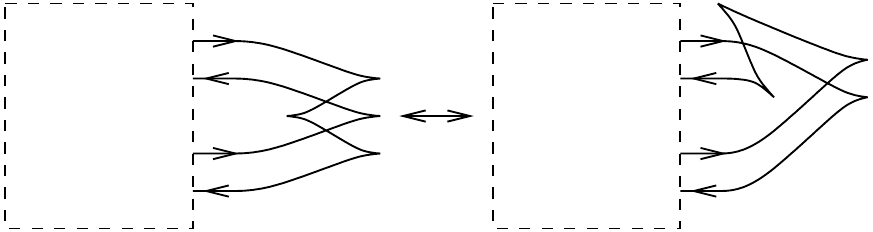}
  \caption{\textbf{Fronts related by an improved $\SZ_+$
      move.} The diagrams are assumed to coincide inside the
      dashed boxes.}
  \label{fig:SZimproved}
\end{figure}

\begin{proof}[Proof of Proposition~\ref{prop:flype-SZ}]
We use grid diagrams as an intermediary between braids, Legendrian
links, and transverse links; see \cite{NgThurston} for the necessary
background. For our purposes, a grid diagram is a link diagram
consisting entirely of non-collinear horizontal and vertical line segments, with
vertical segments crossing over horizontal segments wherever they
intersect.

\begin{figure}
  \centering
  \includegraphics[width=5in]{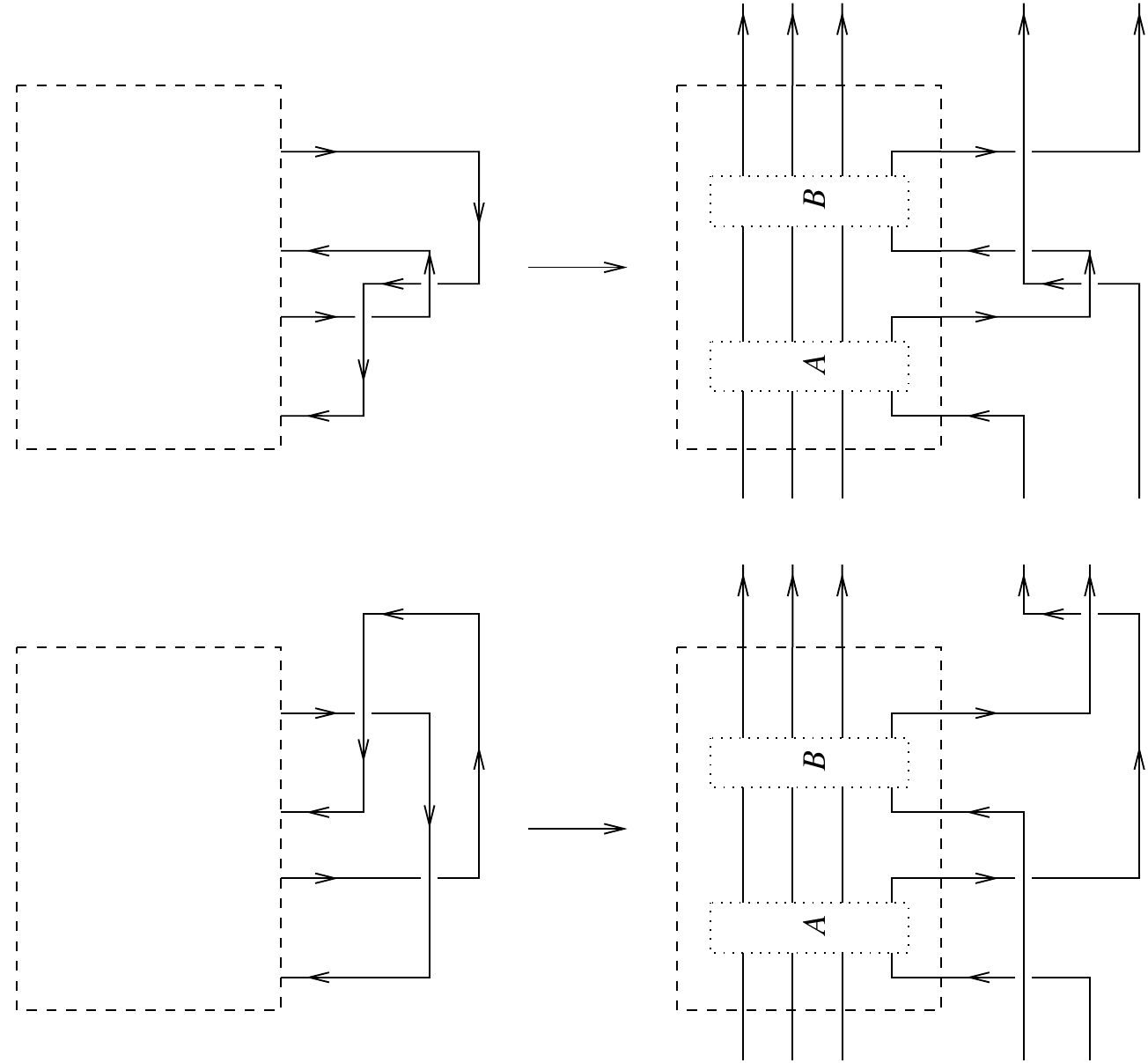}
  \caption{\textbf{Grid and braid diagrams for an improved $\SZ_+$
      move.} Left: grid diagrams related by an improved $\SZ_+$ move;
    right:  the braids obtained from these diagrams.}
  \label{fig:SZtoflype}
\end{figure}

First suppose that two transverse links are the pushoffs
of Legendrian links whose front projections are related by an $\SZ_+$ move.
By applying the Legendrian isotopies shown in Figure~\ref{fig:SZdip},
we may assume that the fronts of the Legendrian links are related by
the ``improved'' $\SZ_+$ move shown in
Figure~\ref{fig:SZimproved}. Now represent these fronts by grid
diagrams as in the left-hand diagrams in
Figure~\ref{fig:SZtoflype}, which agree within the dashed boxes. (Here
we use that grid diagrams can be
viewed as fronts for Legendrian links by rotating them $45^\circ$
counterclockwise and smoothing corners.)

Given a grid diagram $G$ representing a Legendrian link $L$, we can
produce a braid representing the transverse pushoff of $L$, as
follows. Replace any vertical segment oriented downwards (from point $p$ to
point $q$, say) by two half-infinite vertical segments, one pointing
upwards from $p$, and the other pointing upwards to $q$, and again
impose the condition that vertical segments pass over horizontal
segments. The result is the braid, read bottom to top; in the language
of \cite{NgThurston}, this is $B_{\uparrow}(G)$. When we apply this
procedure to the grid diagrams on the left of
Figure~\ref{fig:SZtoflype}, we obtain the braids on the right of
Figure~\ref{fig:SZtoflype}. (Strictly speaking, some of the vertical
segments in the right diagrams should be perturbed slightly to avoid
collinearity.) These braids are of the form
$A\sigma_m^2 B\sigma_m^{-1}$ and $\sigma_m A \sigma_m^{-1} B
\sigma_m$ for some $A,B\in B_m$, and are thus related by a negative
flype (along with conjugation).

\begin{figure}
  \centering
  \includegraphics[width=6in]{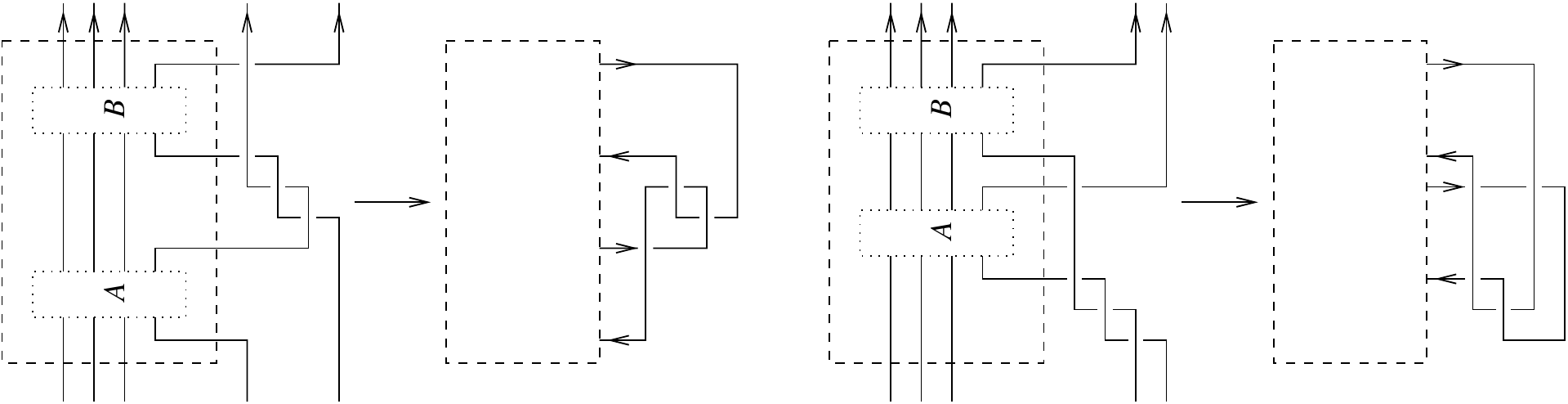}
  \caption{\textbf{Braid and grid diagrams for a negative flype.}
As usual, the diagrams coincide within corresponding dashed boxes.}
  \label{fig:flypetoSZ}
\end{figure}

\begin{figure}
  \centering
  \includegraphics[width=6in]{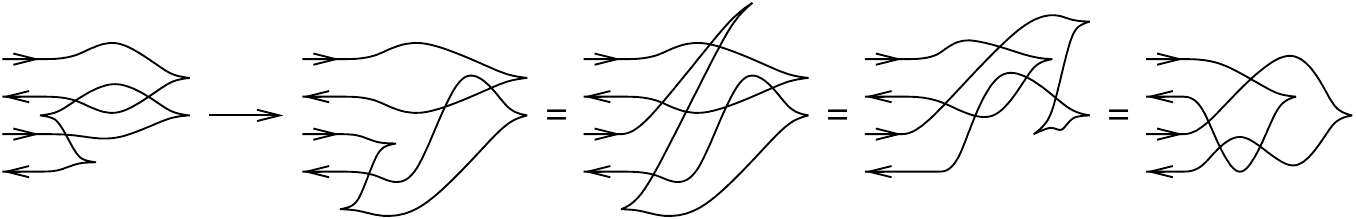}
  \caption{\textbf{An $\SZ_-$ move followed by Legendrian isotopy.}}
  \label{fig:Legisotopy}
\end{figure}

Conversely, suppose that two transverse links are represented by
braids differing by a negative flype $A\sigma_m^{k}B\sigma_m^{-1}
\longleftrightarrow A\sigma_m^{-1}B\sigma_m^{k}$; in the figures, we
assume that $k=3$, but the case of general $k$ is nearly
identical. Draw the braids as in Figure~\ref{fig:flypetoSZ}. Reversing
the procedure from above yields the grid diagrams shown in
Figure~\ref{fig:flypetoSZ}. These grid diagrams, in turn, correspond to
Legendrian links that are related by an $\SZ_+$ move and
Legendrian isotopy. See Figure~\ref{fig:Legisotopy}.
\end{proof}

\section{Some constructions from Khovanov homology}\label{sec:def}

In this section, we collect some standard constructions from Khovanov
homology. Specifically, in Section~\ref{subsec:bar-natan} we recall
the Bar-Natan deformation of the Khovanov complex, and in
Section~\ref{sec:R-maps} we recall the homotopy equivalences of
Bar-Natan complexes induced by Reidemeister moves.

\subsection{The Bar-Natan deformation of Khovanov homology}\label{subsec:bar-natan}
The Bar-Natan deformation of the Khovanov complex
(from~\cite{Bar-kh-tangle-cob}) comes from a $2$-dimensional Frobenius
algebra $V=\ZZ\langle x_-,x_+\rangle$, with multiplication $m$ given
by
\[
x_+\otimes x_+\mapsto x_+ \qquad x_+\otimes x_-\mapsto x_-\qquad
x_-\otimes x_+\mapsto x_- \qquad x_-\otimes x_-\mapsto \Red{x_-},
\]
and comultiplication $\Delta$ by
\[
x_-\mapsto x_-\otimes x_-\qquad x_+\mapsto x_+\otimes x_-+x_-\otimes
x_+\Red{-x_+\otimes x_+}.
\]
Without the terms in \Red{red} this is $H^*(S^2)$, which underlies
Khovanov homology. It is also useful to introduce the variable $\xbar=
x_- - x_+$;\footnote{This definition of $\xbar$ is the negative of the
  one used in~\cite{RS-s}.} with respect to the basis $\{x_-,\xbar\}$ the
multiplication and comultiplication diagonalize to
\begin{equation}\label{eq:Lee-mults}
  \begin{split}
    \xbar\otimes \xbar\stackrel{m}{\longrightarrow}-\xbar\qquad
    \xbar\otimes x_-\stackrel{m}{\longrightarrow}0&\qquad
    x_-\otimes \xbar\stackrel{m}{\longrightarrow}0\qquad
    x_-\otimes x_-\stackrel{m}{\longrightarrow} x_-\\
    \xbar\stackrel{\Delta}{\longrightarrow} \xbar\otimes \xbar&\qquad
    x_-\stackrel{\Delta}{\longrightarrow} x_-\otimes x_-.
  \end{split}
\end{equation}

\begin{figure}
  \centering
  \includegraphics{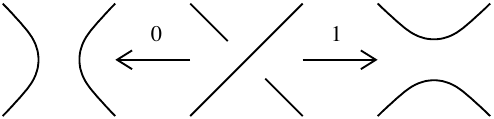}
  \caption{\textbf{Resolutions of a crossing.}}
  \label{fig:resolve-crossing}
\end{figure}

Let $K$ be an oriented link diagram with $n$ crossings. Ordering the
crossings of $K$ induces an identification of the set of complete
resolutions of $K$ with $\{0,1\}^n$. (Our convention for which is the
$0$-resolution and which is the $1$-resolution is given in
Figure~\ref{fig:resolve-crossing}.) Let $\BNcx(K)$ be the chain
complex generated by pairs $(v,x)$ where $v\in\{0,1\}^n$ and $x$ is a
labeling of each circle in the complete resolution $K_v$ corresponding
to $v$ by $x_+$ or $x_-$. (That is, as a vector space, $\BNcx(K)$ is
isomorphic to the Khovanov complex $\KhCx(K)$.) The differential is defined
exactly as for the Khovanov complex, except that we use Bar-Natan's
Frobenius algebra (above) in place of Khovanov's Frobenius algebra.

Gradings will be of some importance. Let $n_-$ (respectively, $n_+$) be
the number of negative (respectively, positive) crossings in $K$.
Given $v\in\{0,1\}^n$ let $|v|=\sum v_i$ denote the weight of $v$. Then
the \emph{homological grading} on $\BNcx(K)$ (or $\KhCx(K)$) is given by
\begin{equation}\label{eq:h-gr}
\gr_h(v,x)=-n_-+|v|.
\end{equation}
The differential on $\BNcx(K)$ (obviously) increases $\gr_h$ by $1$.

The quantum grading on the Khovanov complex becomes a \emph{quantum
  filtration} (or \emph{$q$-filtration}) on the complex
$\BNcx(K)$; it is given by
\begin{equation}\label{eq:q-gr}
\gr_q(v,x)=n_+-2n_-+|v|+\#\{Z\in K_v\mid x(Z)=x_+\} - \#\{Z\in K_v\mid x(Z)=x_-\}.
\end{equation}
That is, up to the normalization $n_+-2n_-$, $\gr_q(v,x)$ is given by
the weight of $v$ plus the number of $x_+$'s in $x$ minus the number
of $x_-$'s in $x$. The differential on $\BNcx(K)$ satisfies
\[
\gr_q(\diff(v,x))\geq \gr_q(v,x).
\]
Let $\Filt_m\BNcx(K)$ denote the part of $\BNcx(K)$ in filtration $\geq m$, i.e.,
\[
\Filt_m\BNcx(K)=\Span\{(v,x)\mid \gr_q(v,x)\geq m\}.
\]
Let $\BNcx^n(K)$ be the part of $\BNcx(K)$ in homological grading $n$, i.e., 
\[
\BNcx^n(K)=\Span\{(v,x)\mid \gr_h(v,x)=n\}.
\]

\begin{figure}
  \centering
  \begin{overpic}[tics=10]{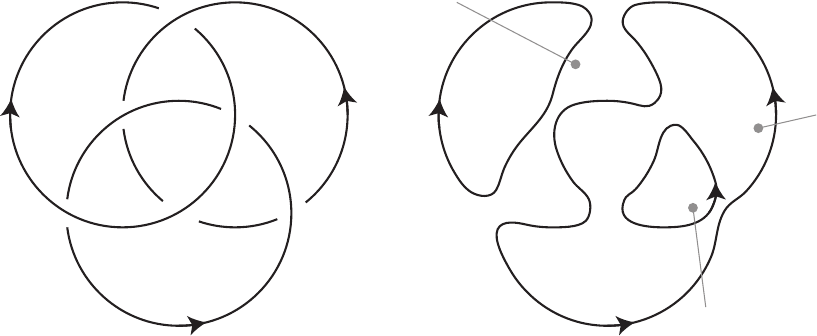}
    \put(59,7){$\xbar$}
    \put(80,11){$x_-$}
    \put(56,32){$x_-$}
  \end{overpic}
  \caption{\textbf{Canonical generators of the Bar-Natan complex.}
    Left: an oriented link. Right: the corresponding oriented
    resolution with the associated Bar-Natan-Lee-Rasmussen-Turner
    generator. (The points $q_C$ and arcs $A$ are indicated in gray.)}
  \label{fig:BN-gens}
\end{figure}

Following Lee~\cite{Lee-kh-endomorphism} (cf.~\cite{Ras-kh-slice}),
Turner showed that the homology of $\BNcx(K)$ is
$(\ZZ\oplus\ZZ)^{\otimes|K|}$ (where $|K|$ denotes the number of components
of $K$)~\cite{Turner-kh-BNseq}. Moreover, the generators of
$H^*(\BNcx(K))$ correspond to orientations of $K$, as follows. Given
an orientation $o$ of $K$, there is a corresponding complete
resolution $K_o$ of $K$, the \emph{oriented resolution}. Each circle
$C$ in $K_o$ inherits an orientation from $K$. Fix a point $p_C$ on
$C$ and let $q_C$ be the result of pushing $p_C$ slightly to the left
of $C$ (with respect to the orientation of $C$). Let $A$ be an arc
from $q_C$ to $\infty$, transverse to all of the circles in $K_o$. If
$A$ crosses an even number of circles then label $C$ by $x_-$; if $A$
crosses an odd number of circles, label $C$ by $\xbar$. See
Figure~\ref{fig:BN-gens}.
This labeling is the cycle $\psi(o)$ of $\BNcx(K)$ corresponding
to $o$. Note that this depended on an arbitrary universal choice:
we could equally well exchange $x_-$ and $\xbar$ in the definition.

\begin{remark}\label{rem:lee-deformation}
  Lee~\cite{Lee-kh-endomorphism} considered another deformation of the
  Khovanov complex, corresponding to the Frobenius structure
  \[
  \begin{split}
  x_+\otimes x_+\stackrel{m}{\mapsto} x_+ \qquad x_+\otimes x_-\stackrel{m}{\mapsto} x_-&\qquad
  x_-\otimes x_+\stackrel{m}{\mapsto} x_- \qquad x_-\otimes x_-\stackrel{m}{\mapsto} \Red{x_+},
  \\
  x_-\stackrel{\Delta}{\mapsto} x_-\otimes x_-+\Red{x_+\otimes x_+}&\qquad x_+\stackrel{\Delta}{\mapsto} x_+\otimes x_-+x_-\otimes
  x_+.
  \end{split}
  \]

  Mackaay-Turner-Vaz showed that over a ring $R$ in which $2$ is
  invertible, the Lee deformation is twist-equivalent (in the sense
  of~\cite{Kho-kh-Frobenius}) to the
  Bar-Natan deformation~\cite{MTV-Kh-s-invts}. In particular, the results below apply to the
  Lee deformation as well, if we work over a ring in which $2$ is
  invertible. This gives no (obvious) additional transverse
  information, but can be useful for studying the Plamenevskaya
  invariant; see Proposition~\ref{prop:olgastabinv}.
\end{remark}

\subsection{The maps induced by Reidemeister moves}\label{sec:R-maps}

For us, a \emph{filtered chain complex with distinguished generators} is a chain
complex generated freely over $\Z$ by a generating set, where each
element of the generating set carries a homological grading $\gr_h$
and a filtration grading $\gr_q$, such that the differential $\diff$
increases the homological grading by $1$ and does not decrease the
filtration grading.
The Bar-Natan chain complex
$\BNcx(K)$ (see Section~\ref{subsec:bar-natan}) is an example of a
filtered chain complex with distinguished generators: the generators in which each circle is labeled by either $x_+$ or $x_-$.

The following is a standard cancellation lemma about filtered chain
complexes with distinguished generators. Since we make repeated use of
it for writing down the Reidemeister maps and for proving locality
properties, we give a proof of it.

\begin{lem}\label{lem:cancellation}
  Let $(C,\diff_C)$ be a filtered chain complex with distinguished
  generators. Let $\al$ and $\be$ be two of the generators such that
  $\Z\langle\al,\be\rangle$ is a subcomplex (respectively, a quotient
  complex) of $C$, 
  where the coefficient $\langle \delta_C\al,\be\rangle$ of $\beta$ in  $\delta_C(\alpha)$ is $\pm1$, and
  $\gr_q(\al)=\gr_q(\be)$. Let $(D,\diff_D)$ be the quotient complex
  (respectively, the subcomplex) of $C$ generated by the remaining
  generators, with $\delta_D$ given by applying $\delta_C$ and then setting to $0$ any generators not in $D$ (respectively restricting $\delta_C$ to $D$). Then the quotient map $f\from C\to D$ (respectively, the
  inclusion map $g\from D\to C$) induces a filtered chain homotopy
  equivalence; that is, there is a filtered chain map $g\from D\to C$
  (respectively, $f\from C\to D$) and filtered homotopies $h\from C\to
  C$, $k\from D\to D$, such that $h\circ \diff_C+\diff_C\circ
  h=\Id_C-g\circ f$ and $k\circ \diff_D+\diff_D\circ k=\Id_D-f\circ
  g$; and in fact one can take $k=0$.
\end{lem}

\begin{proof}
  Assume, after replacing $\be$ by $-\be$ if
  necessary, that $\langle \diff_C \al,\be\rangle=1$.  In the first case,
  when $\Z\langle \al,\be\rangle$ is a subcomplex of $C$, define $g\from
  D\to C$ as
  \[
  g(x)=x-\langle \diff_C x,\be\rangle \al,
  \]
  and in the second case, when $\Z\langle \al,\be\rangle$ is a quotient complex
  of $C$, define $f\from C\to D$ as
  \[
  f(x)=
  \begin{cases}
    0&\text{if $x=\al$,}\\
    -\sum_y\langle\diff_C\al,y\rangle y&\text{if $x=\be$,}\\
    x&\text{otherwise.}
  \end{cases}
  \]
  In either case, define $k=0$ and $h\from C\to C$ as
  \[
  h(x)=
  \begin{cases}
    \al&\text{if $x=\be$,}\\
    0&\text{otherwise.}
  \end{cases}
  \]
  It is straightforward to verify that these maps have the desired properties.
\end{proof}

\begin{remark}
   Lemma~\ref{lem:cancellation} fails
  (even in the non-filtered case) if one does not work with distinguished
  generators: the two step chain complex
  $\Z\overset{2}{\rightarrow}\Z$ admits a quotient map to $\Z/2\Z$ which
  does not induce a chain homotopy equivalence.
\end{remark}

Now, let $K$ and $K'$ be two link diagrams representing isotopic
links. Then $K$ and $K'$ can be related by a sequence of Reidemeister
moves (Figure~\ref{fig:Reid-moves}); and a sequence of
Reidemeister moves connecting $K$ to $K'$ induces a filtered chain
homotopy equivalence between $\BNcx(K)$ and $\BNcx(K')$. We discuss
these filtered chain homotopy equivalences next.

\captionsetup[subfloat]{width=0.1\textwidth}
\begin{figure}
  \centering
  \subfloat[RI${}^-$.\label{subfig:r1n}]{
    \xymatrix@C=0.07\textwidth{
      \vcenter{\hbox{\includegraphics[height=0.3\textheight]{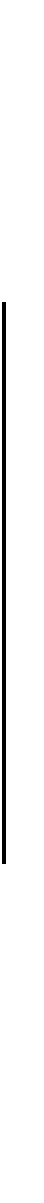}}}\ar@{->}[r]& \vcenter{\hbox{\includegraphics[height=0.3\textheight]{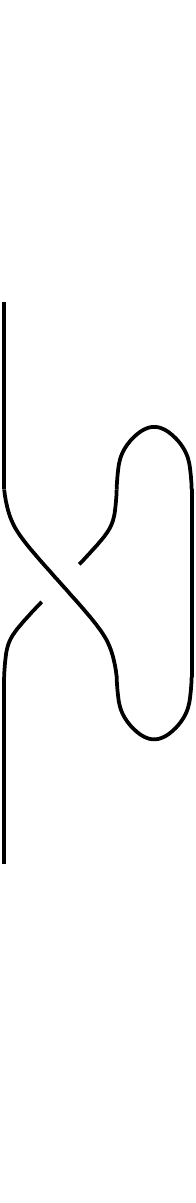}}}
    }
  }
  \hspace{0.03\textwidth}
  \subfloat[RI${}^+$.\label{subfig:r1p}]{
    \xymatrix@C=0.07\textwidth{
      \vcenter{\hbox{\includegraphics[height=0.3\textheight]{reider1b}}}\ar@{->}[r]& \vcenter{\hbox{\includegraphics[height=0.3\textheight]{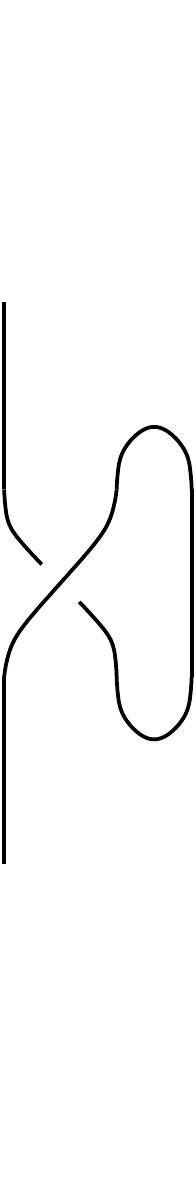}}}
    }
  }
  \hspace{0.03\textwidth}
  \subfloat[RII.\label{subfig:r2}]{
    \xymatrix@C=0.07\textwidth{
      \vcenter{\hbox{\includegraphics[height=0.3\textheight]{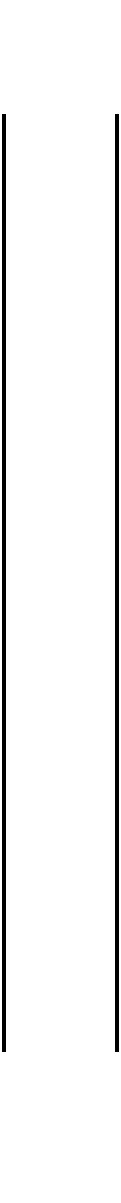}}}\ar@{->}[r]& \vcenter{\hbox{\includegraphics[height=0.3\textheight]{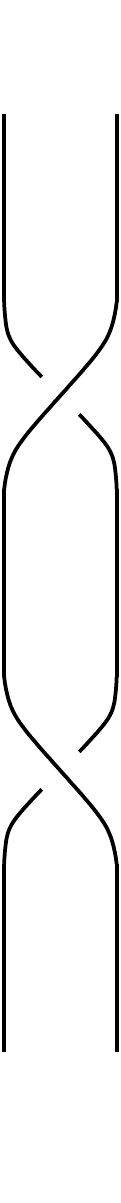}}}
    }
  }
  \hspace{0.03\textwidth}
  \subfloat[RIII.\label{subfig:r3}]{
    \xymatrix@C=0.07\textwidth{
      \vcenter{\hbox{\includegraphics[height=0.3\textheight]{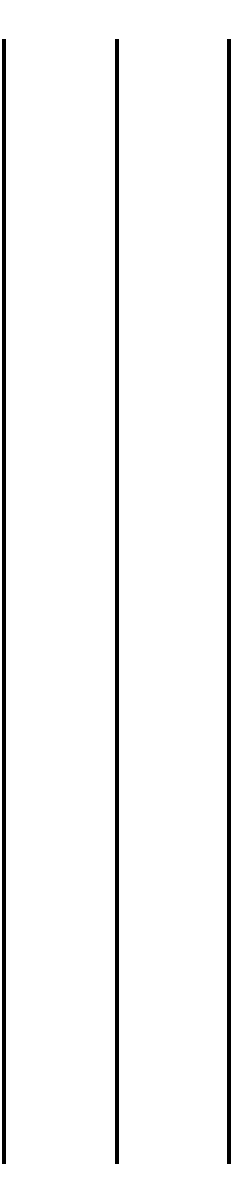}}}\ar@{->}[r]& \vcenter{\hbox{\includegraphics[height=0.3\textheight]{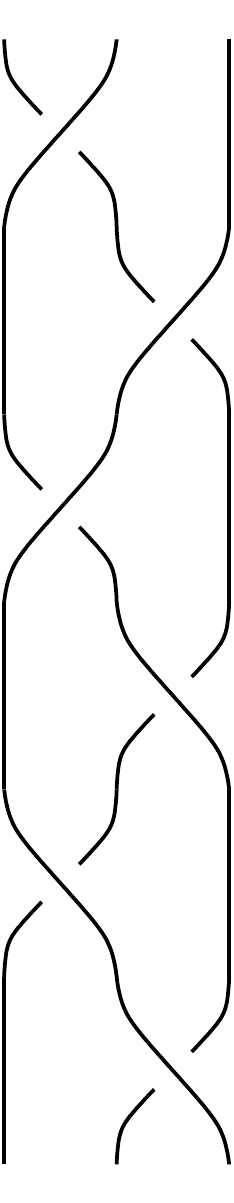}}}
    }
  }
  \caption{\textbf{The four Reidemeister moves.} (a) is the
    Reidemeister I negative stabilization, (b) is the Reidemeister I
    positive stabilization, (c) is the Reidemeister II move, and (d)
    is the braid-like Reidemeister III move. Among these four, (b),
    (c), and (d) may be viewed as the transverse Markov-Reidemeister
    moves.}\label{fig:Reid-moves}
\end{figure}

\subsubsection{Negative stabilization and destabilization}\label{subsec:r1n} Consider the
negative stabilization shown in Figure~\ref{subfig:r1n}. Let $c$
denote the new crossing in the stabilized diagram $K'$, and let
$(K')_0$
(respectively, $(K')_1$) denote the result of replacing $c$ by its
$0$-resolution (respectively, $1$-resolution). Then $(K')_0$ is
isomorphic to the unstabilized diagram $K$, while $(K')_1$ is
isomorphic to the disjoint union of $K$ and an unknot $U_0$. There is a
subcomplex $D$ of $\BNcx(K')$ spanned by the generators in
$\BNcx((K')_1)$ in which $U_0$ is labeled by $x_+$; $D$ can be
identified, preserving the bigrading, with $\BNcx(K)$ by forgetting
the component $U_0$. The map associated to stabilization is the
inclusion map
\[
\BNcx(K)\cong D\to \BNcx(K').
\]
The generators of $\BNcx(K')$ not in $D$ cancel in pairs, via the
arrows $x\mapsto x\otimes x_-$.
Since $D$ is obtained from $\BNcx(K')$ by cancellations of the form
described in Lemma~\ref{lem:cancellation}, the destabilization map is
given by its homotopy inverse.

\subsubsection{Positive stabilization and destabilization}\label{subsec:r1p}  Next, consider
the positive stabilization shown in Figure~\ref{subfig:r1p}. Once
again, let $c$ denote the new crossing in the stabilized diagram $K'$,
and let $(K')_0$ (respectively, $(K')_1$) denote the result of replacing $c$ by
its $0$-resolution (respectively, $1$-resolution). This time, $(K')_1$
is isomorphic to the unstabilized diagram $K$, while $(K')_0$ is
isomorphic to the disjoint union of $K$ and an unknot $U_0$. There is a
subcomplex $E$ of $\BNcx(K')$ spanned by the generators in
$\BNcx((K')_1)$ and the generators in $\BNcx((K')_0)$ in which $U_0$
is labeled by $x_+$; and $\BNcx(K')/E$ is identified with $\BNcx(K)$
by forgetting the component $U_0$. The map associated to
destabilization is the projection
\[
\BNcx(K')\to \BNcx(K')/E\cong\BNcx(K).
\]
Since $E$ can be contracted by cancellations of the form described in
Lemma~\ref{lem:cancellation}, the stabilization map is given by the
homotopy inverse to the destabilization map.

\subsubsection{Reidemeister II}\label{subsec:r2} The maps giving Reidemeister II
invariance can be described quite explicitly. Assume $K'$ is obtained
from $K$ by the move from Figure~\ref{subfig:r2}. Order the two new
crossing in $K'$ and for any $v\in\{0,1\}^2$, let $(K')_v$ denote the
partial resolution of $K'$ at the two new crossings corresponding to
$v$. Choose the ordering of the new crossings so that $(K')_{10}=K$. Then
$(K')_{01}$ has an unknot component, say $U_0$, contained in the
isotopy region. Let $E$ be the subcomplex of $\BNcx(K')$ spanned by the
generators in $\BNcx((K')_{11})$ and the generators in
$\BNcx((K')_{01})$ where $U_0$ is labeled by $x_+$;
therefore, $\BNcx((K')_{10})$ is a subcomplex of $\BNcx(K')/E$ and
we have the following projection and inclusion maps:
\[
\BNcx(K')\stackrel{\pi}{\onto} \BNcx(K')/E\stackrel{\iota}{\hookleftarrow} \BNcx((K')_{10})\cong\BNcx(K).
\]
Both the complexes $E$ and $(\BNcx(K')/E)/\BNcx((K')_{10})$ can be
contracted by cancellations of the form described in
Lemma~\ref{lem:cancellation}, and therefore both $\pi$ and $\iota$
have filtered homotopy inverses, say $\pi^{-1}$ and $\iota^{-1}$.
The map from $\BNcx(K')$ to $\BNcx(K)$ is $\iota^{-1}\circ\pi$ and the map
from $\BNcx(K)$ to $\BNcx(K')$ is $\pi^{-1}\circ\iota$.

\subsubsection{Reidemeister III}\label{subsec:r3}
The maps giving the usual Reidemeister III invariance can also be
written explicitly, but doing so is somewhat tedious. Instead, we
adopt the following indirect argument, which is similar in spirit to
the previous argument. As in~\cite[Section 7.3]{Baldwin-hf-s-seq},
since we have already proved Reidemeister II invariance, it suffices
to prove invariance under the braid-like Reidemeister III move of
Figure~\ref{subfig:r3}.  So, suppose $K$ and $K'$ differ by a
braid-like Reidemeister III move, where $K'$ has $6$ more crossings
than $K$. Order the six new crossings of $K'$, and for any
$v\in\{0,1\}^6$, let $(K')_v$ denote the resolution at the new
crossings corresponding to $v$. Choose the ordering of the new
crossings in $K'$ so that $(K')_{111000}=K$.  It is shown in the proof
of~\cite[Proposition~\ref*{KhSp:prop:RIII}]{RS-khovanov} that there
is a contractible subcomplex $E$ of $\BNcx(K')$ so that
$\BNcx((K')_{111000})$ is a subcomplex of $\BNcx(K')/E$ and
$(\BNcx(K')/E)/\BNcx((K')_{111000})$ is contractible. (Actually,
in~\cite[Proposition~\ref*{KhSp:prop:RIII}]{RS-khovanov}, the
corresponding statement for the Khovanov chain complex $(\KhCx,\diffKh)$, which
is the associated graded object of the filtered Bar-Natan chain
complex $(\BNcx,\diff)$, is proved; however, by looking at the homological
gradings, we see that the subcomplexes and quotient complexes for
$\KhCx$ remain subcomplexes and quotient complexes for $\BNcx$ as
well, and therefore, the statement for the associated graded object
implies the statement for the filtered complex.) This gives a diagram
\[
\BNcx(K')\stackrel{\pi}{\onto} \BNcx(K')/E\stackrel{\iota}{\hookleftarrow} \BNcx((K')_{111000})\cong\BNcx(K).
\]
Furthermore, it is shown in the proof
of~\cite[Proposition~\ref*{KhSp:prop:RIII}]{RS-khovanov} that both the
acyclic complexes $E$ and $(\BNcx(K')/E)/\BNcx((K')_{111000})$ can be
contracted by sequences of elementary cancellations of the form
described in Lemma~\ref{lem:cancellation}. Therefore, both $\iota$ and
$\pi$ are filtered homotopy equivalences, and the homotopy
equivalences $f$ and $g$ are gotten by inverting (up to filtered
homotopy) either $\iota$ or $\pi$.

\subsubsection{Locality of the invariance maps}
  We conclude this subsection with the observation that the
  Reidemeister maps are local in a particular sense, a fact that we
  will need in Section~\ref{sec:flypes}. (This is well-known. Other
  versions of locality are exploited, for instance, in~\cite{Bar-kh-tangle-cob}
  and~\cite{Kho-kh-cob}.)

\begin{proposition}\label{prop:local}
  Suppose that $K_1$ is a link diagram and $T_1\subset K_1$ is a
  tangle. Let $T_2$ be a tangle diagram representing a tangle isotopic
  to $T_1$, and let $K_2=(K_1\setminus T_1)\cup T_2$ be the result of
  replacing $T_1$ with $T_2$ in $K_1$.  Let $f\co \BNcx(K_1)\to
  \BNcx(K_2)$ and $g\co \BNcx(K_2)\to \BNcx(K_1)$ be the filtered maps
  induced by an isotopy from $T_1$ to $T_2$, as above, and let $h$ and
  $g$ denote the filtered homotopies,
  so that $\Id-g\circ f = \diff\circ h + h\circ \diff$ and $\Id-f\circ
  g=\diff\circ k+k\circ \diff$.

  Order the crossings of $K_i$ ($i=1,2$) so that the crossings in
  $K_i\setminus T_i$ come before the crossings in $T_i$; and so that
  the orderings of the crossings in $K_1\setminus T_1=K_2\setminus
  T_2$ agree. Suppose that there are $n$ crossings in $K_i\setminus
  T_i$, and $n_i$ crossings in $T_i$.

  Then:
  \begin{itemize}
  \item Given a generator $((u,v),x)\in \KhCx(K_1)$, with
    $(u,v)\in\{0,1\}^n\times\{0,1\}^{n_1}\cong \{0,1\}^{n+n_1}$,
    $f((u,v),x)$ lies over the sub-cube $\{u\}\times\{0,1\}^{n_2}$ of
    $\{0,1\}^{n+n_2}$; and $h((u,v),x)$ lies over the sub-cube
    $\{u\}\times\{0,1\}^{n_1}$ of $\{0,1\}^{n+n_1}$.

    Moreover, if $C$ is a circle in the resolution $(K_1)_{(u,v)}$
    disjoint from $T_1$, and $C'$ is the corresponding circle in some
    resolution $(K_2)_{(u,w)}$, then the label $f((u,v),x)(C')$ of $C'$
    in $f((u,v),x)$ and the label $h((u,v),x)(C)$ of $C$ in
    $h((u,v),x)$ are the same as the label $x(C)$ of $C$ in $x$.
  \item Given a generator $((u,w),y)\in \KhCx(K_2)$, with
    $(u,w)\in\{0,1\}^n\times\{0,1\}^{n_2}\cong \{0,1\}^{n+n_2}$,
    $g((u,w),y)$ lies over the sub-cube $\{u\}\times\{0,1\}^{n_1}$ of
    $\{0,1\}^{n+n_2}$; and $k((u,w),y)$ lies over the sub-cube
    $\{u\}\times\{0,1\}^{n_2}$ of $\{0,1\}^{n+n_2}$.

    Moreover, if $C$ is a circle in the resolution $(K_2)_{(u,w)}$
    disjoint from $T_2$, and $C'$ is the corresponding circle in some
    resolution $(K_2)_{(u,w)}$, then the label $g((u,w),y)(C')$ of $C'$
    in $g((u,w),y)$ and the label $k((u,w),y)(C)$ of $C$ in
    $k((u,w),y)$ are the same as the label $y(C)$ of $C$ in $y$.
  \end{itemize}
\end{proposition}
\begin{proof}
  This is immediate from the form of the maps given in
  Sections~\ref{subsec:r1n}--\ref{subsec:r3}; the key point is that
  the subcomplexes and quotient complexes, and the arrows one cancels
  (Lemma~\ref{lem:cancellation}) to prove that the inclusion and
  projection maps are filtered homotopy equivalences, are all given
  locally.
\end{proof}

\section{The filtered Plamenevskaya invariant}\label{sec:filt-plam}
In this section we define the (filtered extension of the)
Plamenevskaya invariant (Section~\ref{sec:def-filt}) and prove it is
invariant under transverse isotopies
(Section~\ref{sec:invariance}). We then make some observations on its
behavior under negative stabilization (Section~\ref{sec:neg-stab})
before proving our main theorem, invariance under flypes and $\SZ$ moves
(Section~\ref{sec:flypes}). After this, we make a few further observations
guaranteeing that the filtered Plamenevskaya invariant does no better
than the classical invariants at distinguishing transverse
representatives of low-crossing knots (Section~\ref{sec:sad}).

\subsection{The definition of the invariant}\label{sec:def-filt}
Let $K$ be a transverse link in $S^3$ (with respect to the
standard contact structure $\xi_\std$), presented as the closure of an
(oriented) braid. Abusing notation slightly, we will also use $K$ to
denote the corresponding link diagram.

\begin{definition}
  Recall from Section~\ref{subsec:bar-natan} that associated to each
  orientation $o$ of $K$ is a cycle $\psi(o)$ in $\BNcx(K)$.  The
  \emph{positive (respectively, negative) filtered Plamenevskaya
    invariant} of $K$ is the generator $\ppsi(K)\coloneqq \psi(o)$
  (respectively, $\npsi(K)\coloneqq \psi(-o)$) corresponding to the
  usual orientation of $K$ as a transverse link (respectively, the
  opposite orientation to the usual one).
  Since most of the results in this paper will apply to both
  $\ppsi(K)$ and $\npsi(K)$, we will use
  $\pnpsi(K)$ to denote either $\ppsi(K)$ or $\npsi(K)$. The
  invariant $\pnpsi(K)$ lies in homological grading
  \[
  \gr_h(\pnpsi(K))=-n_-+n_-=0.
  \]
  With respect to the $q$-filtration, the lowest filtered part of
  $\psi(\pm o)$ is when each circle is labeled by $x_-$;
  this lies in filtration
  \[
  \gr_q(\pnpsi(K))=n_+-n_--m=\SL(K).
  \]
  (Here $m$ denotes the braid index.)
  We thus regard $\pnpsi(K)$ as an
  element of $\Filt_{\SL(K)}\BNcx^0(K).$
\end{definition}

The sense in which $\pnpsi(K)$ is an invariant of the transverse
isotopy class of $K$ is made precise in
Theorem~\ref{thm:filt-inv}. Note that it follows from (the argument
in)~\cite{Lee-kh-endomorphism} that $\pnpsi(K)\in \BNcx^0(K)$ is a
cycle; this can also be seen directly from the fact that the only differentials from the oriented resolution correspond to merge cobordisms, and $m(\xbar\otimes x_-)=m(x_-\otimes \xbar)=0$, cf.~Formula~\eqref{eq:Lee-mults}.

The filtered Plamenevskaya invariant induces a number of auxiliary
invariants. For any $p\leq 0< q$, let
\[
\pnpsi_{p,q}(K)\in \Filt_{\SL(K)+2p}\BNcx^0(K)/\Filt_{\SL(K)+2q}\BNcx^0(K)
\]
denote the image of $\pnpsi(K)$ under the obvious chain map
$\Filt_{\SL(K)}\BNcx\to
\Filt_{\SL(K)+2p}\BNcx/\Filt_{\SL(K)+2q}\BNcx$, and let
\[
[\pnpsi_{p,q}(K)]\in
H^0(\Filt_{\SL(K)+2p}\BNcx(K)/\Filt_{\SL(K)+2q}\BNcx(K))
\]
denote the image of $\pnpsi_{p,q}$ in the homology.  Note:
$\pnpsi_{0,\infty}(K)=\pnpsi(K)$; $[\ppsi_{-\infty,\infty}]$ and
$[\npsi_{-\infty, \infty}]$ are two of the generators of $H^*(\BNcx)$
corresponding to the orientations $o$ and $-o$ respectively; and the
element
\[
\pnpsi_{0,1}(K)\in\Filt_{\SL(K)}\BNcx^0(K)/\Filt_{\SL(K)+2}\BNcx^0(K)=\KhCx^{0,\SL(K)}(K)
\]
is the transverse invariant defined by
Plamenevskaya~\cite{Plam-06-KhTrans}. In particular,
$\ppsi_{0,1}(K)=\npsi_{0,1}(K)$. Indeed, since the lowest-filtration
parts of $\ppsi(K)$ and $\npsi(K)$ are the same (every circle
decorated by $x_-$), the difference element
$\dpsi(K)\coloneqq \ppsi(K)-\npsi(K)$ is (a
cycle) in
$\Filt_{\SL(K)+2}\BNcx(K)$. For $p\leq 1<q$, let
$\dpsi_{p,q}(K)$ denote the image of $\dpsi(K)\in \Filt_{\SL(K)+2}\BNcx(K)$ in the subquotient
$\Filt_{\SL(K)+2p}\BNcx/\Filt_{\SL(K)+2q}\BNcx$. Then we have
$\dpsi_{1,\infty}(K)=\dpsi(K)\in \Filt_{\SL(K)+2}\BNcx(K)$, while
$\dpsi_{0,\infty}(K)=\ppsi(K)-\npsi(K)\in \Filt_{\SL(K)}\BNcx(K)$ is determined by the invariants $\ppsi(K)$ and $\npsi(K)$.

As an aside, note that the Rasmussen
invariant $s(K)$ has the following description:
\[
s(K)=\SL(K)-1+2(\min\set{q}{[\pnpsi_{-\infty,q}(K)]\neq 0}).
\]
In particular, this implies the upper bound on the self-linking number first
observed by Plamenevskaya~\cite[Section 7]{Plam-06-KhTrans}:
\[
s(K)\geq\SL(K)-1+2=\SL(K)+1.
\]

\subsection{Invariance under transverse isotopies}\label{sec:invariance}
In this subsection we prove transverse invariance of $\pnpsi(K)$ and
$\dpsi(K)$. We start with $\pnpsi(K)$:
\begin{theorem}\label{thm:filt-inv}
  Suppose that $K$ and $K'$ are diagrams for closed braids, and that
  the corresponding transverse links are transversely isotopic. Then
  the filtered homotopy equivalence $f\co \BNcx(K)\to \BNcx(K')$,
  induced by some sequence of transverse Markov moves connecting $K$
  to $K'$, satisfies
  \[
  f(\pnpsi(K))=\pm\pnpsi(K')+\diff \phi
  \]
  (for one choice of $+$ or $-$), where
  $\phi\in\Filt_{\SL(K)}\BNcx^{-1}(K')$. Moreover, the filtered
  homotopy equivalence $f$ and the element $\phi$ can be chosen to be
  local in the sense of Proposition~\ref{prop:local}.
\end{theorem}

\begin{corollary}
  The image $[\pnpsi_{p,q}]$ of $\pnpsi_{p,q}$ in homology
  $H^0\bigl(\Filt_{\SL(K)+2p}\BNcx(K)/\Filt_{\SL(K)+2q}\BNcx(K)\bigr)$
  is an invariant of the transverse link $K$.
\end{corollary}

The case $p=0,q=1$ is~\cite[Theorem 2]{Plam-06-KhTrans}. The general case is not
significantly different, but we spell out the proof.

\begin{proof}[Proof of Theorem~\ref{thm:filt-inv}]
  By the transverse Markov
  theorem~\cite{Wri-knot-transMarkov,OS-knot-transMarkov},
  it suffices to
  prove that $\pnpsi(K)$ is invariant under (braid-like) Reidemeister II
  and III moves and positive stabilizations and destabilizations (see
  Figures~\ref{subfig:r1p}--\ref{subfig:r3}). The relevant maps are
  described in Section~\ref{subsec:r1p}--\ref{subsec:r3}. We discuss
  each move briefly in turn.

  For positive destabilization, with notation as in
  Section~\ref{subsec:r1p}, note that $\pnpsi(K)$ lies in
  $\BNcx((K')_0)$, and has $U_0$ labeled by either $x_-$ or $\xbar$
  (and the other components of the oriented resolution labeled
  alternately by $\xbar$ and $x_-$). In either case, $\pnpsi(K')$
  survives to $\BNcx(K')/D$ and is identified with $\pnpsi(K')$ under
  the identification $\BNcx(K')/D\cong \BNcx(K)$.  Positive
  stabilization is the homotopy inverse to positive destabilization,
  and hence also 
  preserves $\pnpsi$ (up to the boundary of an element of $\Filt_{\SL(K)}\BNcx(K)$).

  Reidemeister II and III are easier. For Reidemeister II, with
  notation as in Section~\ref{subsec:r2}, $\pnpsi(K')$ lies in
  $\BNcx((K')_{10})$, and so $\pi(\pnpsi(K'))$ is exactly
  $\iota(\pnpsi(K))$. Since the two filtered homotopy equivalences are
  obtained by inverting either $\pi$ or $\iota$ (up to homotopy), they
  respect $\psi$. Similarly, for Reidemeister III, with notation as in
  Section~\ref{subsec:r3}, $\pnpsi(K')$ lies in
  $\BNcx((K')_{111000})$, so again $\pi(\pnpsi(K'))=\iota(\pnpsi(K))$,
  and hence the filtered homotopy equivalences respect $\psi$.

  The fact that the map $f$ and element $\phi$ can be chosen to be
  local follows from Proposition~\ref{prop:local}.
\end{proof}

\begin{remark}\label{rem:linear-combinations}
The above proof of Theorem~\ref{thm:filt-inv} actually shows more
generally that maps induced by transverse isotopy preserve any linear
combination $\al\ppsi+\be\npsi$ (not just $\ppsi$ and $\npsi$), up to
a boundary in $\Filt_{\SL(K)}\BNcx$ and an overall sign.
\end{remark}
Now we turn to the invariant $\dpsi(K)$. By
Remark~\ref{rem:linear-combinations},
$\ppsi(K)-\npsi(K)=\dpsi_{0,\infty}(K)$ is an invariant in
$\Filt_{\SL(K)}\BNcx(K)$; but in fact more is true:

\begin{theorem}\label{thm:diff-is-invt}
  The element $\dpsi(K)\in\Filt_{\SL(K)+2}\BNcx(K)$ is an invariant of
  the transverse link $K$; more precisely, if
  $f\from\BNcx(K)\to\BNcx(K')$ is the filtered homotopy equivalence
  induced by some sequence of transverse Markov moves connecting
  closed braids $K$ and $K'$, then
  \[
  f(\dpsi(K))=\pm\dpsi(K')+\diff\phi
  \]
  for some $\phi\in\Filt_{\SL(K)+2}\BNcx^{-1}(K')$.
\end{theorem}
\begin{proof}
  As in the proof of Theorem~\ref{thm:filt-inv}, we check
  invariance under positive (de)stabilization and braid-like
  Reidemeister II and III moves. For positive destabilization, we
  verified in the proof of Theorem~\ref{thm:filt-inv} that $\ppsi(K')$
  (respectively, $\npsi(K')$) maps to $\ppsi(K)$ (respectively,
  $\npsi(K)$) on the nose. It follows that the destabilization map
  takes their difference $\dpsi(K')$ to $\dpsi(K)$. The
  stabilization map is a filtered homotopy inverse to the
  destabilization map, and hence takes $\dpsi(K)$ to $\dpsi(K')$ up to
  the boundary of an element of $\Filt_{\SL(K)+2}\BNcx(K)$. This proves invariance under positive
  (de)stabilization.  The proof of Reidemeister II and III invariance
  is exactly as in Theorem~\ref{thm:filt-inv} (with $\dpsi$ in place
  $\pnpsi$).
\end{proof}

\begin{remark}
  The reason for the $\pm$ signs in Theorems~\ref{thm:filt-inv}
  and~\ref{thm:diff-is-invt} is perhaps not obvious: the Reidemeister
  maps we wrote down respect $\pnpsi$ (respectively, $\dpsi$) exactly,
  not just up to sign. The $\pm$ signs arise because the maps on
  Khovanov homology (and indeed, the Khovanov homology groups
  themselves) are only well-defined up to an overall sign.
\end{remark}

\subsection{Behavior under negative stabilization}\label{sec:neg-stab}
\begin{proposition}\label{prop:neg-stab}
  Let $K$ be a closed braid in $S^3$ and let $K'$ be the result of negatively
  stabilizing $K$ once, so $\SL(K')=\SL(K)-2$. Let
  $\pnpsi(K)\in\Filt_{\SL(K)}\BNcx(K)$ and
  $\pnpsi(K')\in\Filt_{\SL(K')}\BNcx(K')$ be the corresponding
  filtered Plamenevskaya invariants and let $f\co \BNcx(K)\to \BNcx(K')$
  and $g\co \BNcx(K')\to\BNcx(K)$ be the filtered chain maps giving
  Reidemeister I invariance (cf.~Section~\ref{subsec:r1n}). Then there
  are elements $\theta\in\Filt_{\SL(K')}\BNcx^{-1}(K)$ and
  $\theta'\in\Filt_{\SL(K')}\BNcx^{-1}(K')$, local in the sense of
  Proposition~\ref{prop:local}, such that
  \begin{align*}
  g(\pnpsi(K'))&=\pm\pnpsi(K)+\diff\theta\text{ in }\Filt_{\SL(K')}\BNcx(K)
  \shortintertext{and}
  f(\pnpsi(K))&=\pm\pnpsi(K') +\diff\theta'\text{ in }\Filt_{\SL(K')}\BNcx(K').
  \end{align*}
\end{proposition}
\begin{proof}
  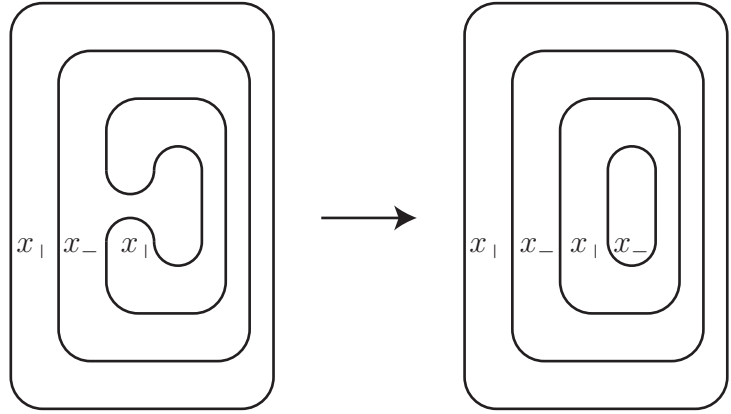
\begin{figure}
    \centering
    \begin{overpic}[tics=10]{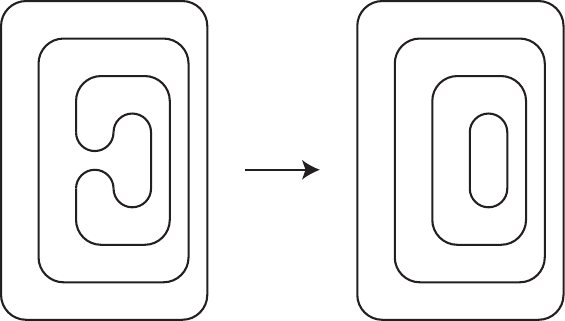}
      \put(84,22){$x_-$}
      \put(78,22){$\xbar$}
      \put(71,22){$x_-$}
      \put(64,22){$\xbar$}
      \put(15.5,22){$\xbar$}
      \put(7.5,22){$x_-$}
      \put(1,22){$\xbar$}
    \end{overpic}
    \caption{\textbf{Invariance under negative stabilization.} The
      generator $\theta'$ is shown on the left, and $\pnpsi(K')$ on the
      right, both in the case that $x=x_-$.}
    \label{fig:neg-stab-inv}
  \end{figure}
  With notation as in Section~\ref{subsec:r1n}, $\pnpsi(K')$ lies in
  $\BNcx((K')_1)$ but not in the subcomplex $D$. Let
  $x\in\{x_-,\xbar\}$ denote the label of $U_0$ in $\pnpsi(K')$, and
  let $\theta'\in\BNcx((K')_0)$ be the element which labels the
  circles alternately by $x_-$ and $\xbar$, starting by labeling the
  inner-most circle oppositely from $x$. (So, $\theta'$ and $x$ give
  the same labels to circles occurring both in $(K')_0$ and $(K')_1$;
  see Figure~\ref{fig:neg-stab-inv}.)
  Assume the crossings in $K'$ are ordered so that the new crossing
  $c$ is the first, and the other crossings are ordered as in
  $K$. Then $\pnpsi(K')-\diff\theta'$ lies in $D$; and under the
  identification $D\cong \BNcx(K)$, it is identified with $\psi(K)$
  (respectively, $-\psi(K)$) if $x$ equals $x_-$ (respectively,
  $\xbar$). The element $\theta'$ lies in
  filtration $\SL(K')$, verifying that $f$ has the specified
  form. (Note that $\theta'$ is local in the sense of
  Proposition~\ref{prop:local}.) It follows that the filtered homotopy
  inverse $g$ to $f$ also respects $\pnpsi$, as an element of
  $\Filt_{\SL(K')}\BNcx$; the fact that the cancellations are local
  implies that we can take $\theta$ to be local, as well.
\end{proof}

\begin{remark}\label{rem:neg-stabilization-linear}
  In a similar vein to Remark~\ref{rem:linear-combinations}, observe
  that (up to an overall sign, and relative boundary in
  $\Filt_{\SL(K')}\BNcx$) the negative stabilization map sends the
  linear combination $\al\ppsi(K)+\be\npsi(K)$ to
  $\al\ppsi(K')-\be\npsi(K')$, and the negative destabilization map
  sends the linear combination $\al\ppsi(K')+\be\npsi(K')$ to
  $\al\ppsi(K)-\be\npsi(K)$.
\end{remark}

\begin{corollary}\label{cor:pre-Jake-preserved}
  If transverse links $K$ and $K'$ become transversely isotopic after
  performing $k$ transverse (negative) stabilizations then
  $\pnpsi_{p,\infty}(K)=\pm\pnpsi_{p,\infty}(K')$ for any $p\leq
  -k$. That is, if $f\from\BNcx(K)\to\BNcx(K')$ is the map associated
  to an isotopy from $K$ to $K'$ corresponding to $k$ transverse
  stabilizations, then a transverse isotopy, and then $k$ transverse
  destabilizations, then $f(\pnpsi(K))=\pm\pnpsi(K')+\diff\theta$ for
  some $\theta\in\Filt_{\SL(K)-2k}\BNcx^{-1}(K')$.
\end{corollary}
\begin{proof}
  This is immediate from Theorem~\ref{thm:filt-inv} (which covers
  transverse isotopies) and Proposition~\ref{prop:neg-stab} (which
  covers negative stabilizations and destabilizations).
\end{proof}

The following was proved by Rasmussen~\cite{Ras-kh-slice} in the case
of the Lee deformation; we will use it in Section~\ref{sec:flypes}:
\begin{corollary}\label{cor:Jake-preserved}
  Reidemeister moves preserve filtered Plamenevskaya generators up to
  homotopy. That is, let $K$ and $K'$ be diagrams for isotopic links
  and let $f\co \BNcx(K)\to \BNcx(K')$ be the map of Bar-Natan complexes
  induced by a sequence of Reidemeister moves connecting $K$ and
  $K'$. Then $f(\pnpsi(K))=\pm\pnpsi(K')+\diff \theta$ for some
  $\theta\in\BNcx^{-1}(K')$.
\end{corollary}
\begin{proof}
  This follows immediately from
  Corollary~\ref{cor:pre-Jake-preserved}. Note that in general we do
  not have any control over the filtration level in which $\theta$ lies.
\end{proof}

\begin{corollary}\label{cor:destabilize-k-times}
If a transverse link $K$ can be negatively destabilized $k$ times,
then $[\pnpsi_{0,k}(K)]=0$.
\end{corollary}

\begin{proof}
Let $K'$ be a transverse link so that $S^k(K')=K$. Let
$f\from\BNcx(K')\to\BNcx(K)$ be the filtered chain map corresponding
to $k$ negative stabilizations. Proposition~\ref{prop:neg-stab}
ensures the existence of some $\theta\in\Filt_{\SL(K)}\BNcx(K)$ so that
\[
f(\pnpsi(K'))=\pm\pnpsi(K)+\diff\theta.
\]
However, since $f$ is a filtered map, $f(\pnpsi(K'))$ lies in
$\Filt_{\SL(K')}\BNcx(K)=\Filt_{\SL(K)+2k}\BNcx(K)$. Therefore, letting
$\theta_{0,k}$ denote the image of $\theta$ in
$\Filt_{\SL(K)}\BNcx(K)/\Filt_{\SL(K)+2k}\BNcx(K)$, we have
\[
\pnpsi_{0,k}(K)=\mp\diff(\theta_{0,k})
\]
and hence $[\pnpsi_{0,k}(K)]=0$.
\end{proof}

\begin{remark}
Note that, by the $s$-invariant upper bound on self-linking number
\cite{Plam-06-KhTrans}, we get
\[
\SL(K')=\SL(K)+2k\leq s(K)-1=\SL(K)-2+2(\min\set{q}{[\pnpsi_{-\infty,q}(K)]\neq
  0}),
\]
or
\[
\min\set{q}{[\pnpsi_{-\infty,q}(K)]\neq 0}\geq k+1.
\]
Corollary~\ref{cor:destabilize-k-times} furnishes a (possibly)
stronger inequality:
\[
\min\set{q}{[\pnpsi_{0,q}(K)]\neq 0}\geq k+1.
\]
\end{remark}

We conclude this subsection with an observation regarding the original
Plamenevskaya invariant $\psi_{0,1}$ and the
property~(\ref{flype2}) from
Proposition~\ref{prop:transverse-equivalence-list}.

\begin{proposition}\label{prop:olgastabinv}
  Let $K$ and $K'$ be transverse links so that $S(K)=S(K')$. Fix a
  ring $R$ in which $2$ is invertible.  Let
  $f\from\KhCx(K;R)\to\KhCx(K';R)$ be the chain map on the Khovanov
  chain complexes over $R$ associated to an isotopy from $K$ to $K'$
  corresponding to a single negative stabilization, then a transverse
  isotopy, and then a single negative destabilization. Then
  $f(\psi_{0,1}(K))=\al\psi_{0,1}(K')+\diffKh(\phi)$ for some
  $\phi\in\KhCx^{-1,\SL(K)}(K';R)$ and some unit $\al$ in $R$.
\end{proposition}

\begin{proof}
Consider the Lee filtered chain complex $(\LeeCx,\diffLee)$ from
\cite{Lee-kh-endomorphism}. Let $\psi_{\mathit{Lee}}(K)$ be the
Lee generator corresponding the usual orientation of $K$; it is a
cycle lying in filtration $\Filt_{\SL(K)}\LeeCx(K)$, and its lowest
filtration term is also $\psi_{0,1}(K)$.

However, by~\cite[Proposition 2.2]{MTV-Kh-s-invts}, over the ring $R$,
the Bar-Natan filtered complex
$(\BNcx,\diff)$ is twist equivalent to the Lee filtered complex
$(\LeeCx,\diffLee)$ (see Remark~\ref{rem:lee-deformation}). Therefore,
over the ring $R$, Corollary~\ref{cor:pre-Jake-preserved} also holds
for the Lee complex. That is, we have
\[
f(\psi_{\mathit{Lee}}(K))=\al\psi_{\mathit{Lee}}(K')+\diffLee(\theta)
\]
for some $\theta\in\Filt_{\SL(K)-2}\LeeCx(K')$ and some unit $\al\in
R$.

Since $\diffLee$ preserves the quantum
grading mod $4$, there is a quotient map
\[
\Filt_{\SL(K)-2}\LeeCx\onto\Kh^{*,\SL(K)}.
\]
Letting $\phi\in\Kh^{-1,\SL(K)}$ denote the image of $\theta$ under
this quotient map, we get the desired identity:
\[
f(\psi_{0,1}(K))=\al\psi_{0,1}(K')+\diffKh(\phi).\qedhere
\]
\end{proof}

\begin{remark}\label{rmk:filtered}
Although both the original Plamenevskaya invariant $\psi_{0,1}$ and
the filtered version $\psi$ are invariant under negative flypes (see
Theorem~\ref{thm:flype-inv} below), Proposition~\ref{prop:olgastabinv}
suggests one way in which we might be hopeful that the filtered
invariant could be effective even if the original invariant were
not. To our knowledge, it is conceivable that there could be
transverse knots that
become the same after one stabilization, but are not related by a
sequence of negative flypes; compare conditions (\ref{flype1}) and
(\ref{flype2}) in
Proposition~\ref{prop:transverse-equivalence-list}. Such knots are
indistinguishable using $\psi_{0,1}$ by
Proposition~\ref{prop:olgastabinv}, but could possibly be
distinguished using the filtered invariant $\psi$.
\end{remark}

\subsection{Invariance under negative flypes}\label{sec:flypes}
In this subsection we prove that the filtered Plamenevskaya invariant
is unchanged by negative flypes (Theorem~\ref{thm:flype-inv}), and
that the difference invariant $\dpsi$ is unchanged by a subclass of
negative flypes (Theorem~\ref{thm:diff-flype-inv}).

\begin{theorem}\label{thm:flype-inv}
  The filtered Plamenevskaya invariant is invariant under negative
  flypes. That is, if $K$ and $K'$ are related by a negative flype and
  $f\co \BNcx(K)\to\BNcx(K')$ is the map associated to the sequence of
  Reidemeister moves~\eqref{eq:flype-moves} then
  $f(\pnpsi(K))=\pm\pnpsi(K')+\diff \phi$ (for one choice of $+$ or $-$), where
  $\phi\in\Filt_{\SL(K)}\BNcx(K')$.
\end{theorem}

\begin{corollary}
  The filtered Plamenevskaya invariant is invariant under flype and
  $\SZ$ equivalence.
\end{corollary}

Before proving Theorem~\ref{thm:flype-inv}, we fix some notation.
For the rest of the subsection, fix braids $A$ and $B$ as in
Definition~\ref{def:flype} and let $K$ (respectively, $K'$) be the
braid closure of $A\sigma_m^{k}B\sigma_m^{-1}$ (respectively,
$A\sigma_m^{-1}B\sigma_m^{k}$).

Let $n+|k|+1$ be the number of crossings of $K$. Order the crossings
of $K$ and $K'$ so that the first $n$ crossings lie in $A\cup B$ and
the $(n+1)^{\st}$ crossing is $\sigma_m^{-1}$. Let $u\in\{0,1\}^n$
correspond to the oriented resolution of $A\cup B$; that is, $u$ is
$0$ at each $\sigma_i$ and $1$ at each $\sigma_i^{-1}$.

Let $f\co\BNcx(K)\to\BNcx(K')$ be the filtered chain homotopy
equivalence induced by the sequence of moves~(\ref{eq:flype-moves}). By Corollary~\ref{cor:Jake-preserved},
\[
f(\pnpsi(K))=\pm\pnpsi(K') + \diff\phi
\]
for some $\phi\in\BNcx^{-1}(K')$. We want to show that
$\phi\in\Filt_{\SL(K)}\BNcx(K')$.

\begin{lemma}\label{lem:live-over}
  With the notations from above, the following statements hold.
  \begin{enumerate}
  \item If $k\geq 0$, the elements $\pnpsi(K)$ and $\pnpsi(K')$ lie
    over the vertex $u\times(1,0,0,\dots,0)\in\{0,1\}^{n+|k|+1}$;
    otherwise, they lie over the vertex
    $u\times(1,1,1,\dots,1)\in\{0,1\}^{n+|k|+1}$.
  \item If $k\geq 0$, the element $f(\pnpsi(K))$ lies over the
    vertices $u\times(\epsilon_1,\dots,\epsilon_{|k|+1})$ where exactly
    one of the $\epsilon_i$ is $1$ (and the rest are $0$); otherwise,
    it lies over the vertex $u\times (1,\dots,1)$.
  \item If $k\geq 0$, the element $\phi$ lies over the vertex
    $u\times(0,0,0,\dots,0)\in \{0,1\}^{n+|k|+1}$; otherwise, it lies
    over the vertices $u\times (\epsilon_1,\dots,\epsilon_{|k|+1})$
    where exactly one of the $\epsilon_i$ is $0$ (and the rest are
    $1$).
  \end{enumerate}
\end{lemma}
\begin{proof}
  The first part of the statement is immediate from the definitions:
  the oriented resolution corresponds to taking the $1$-resolution at
  each $\sigma_i$ and the $0$-resolution at each $\sigma_i^{-1}$. The
  second follows from the first part, locality of the map $f$
  (Proposition~\ref{prop:local}) and the fact that $f$ respects the
  homological grading. The third part follows from locality
  (Proposition~\ref{prop:local}) and the fact that $\phi$ lies in
  homological grading $1$ lower than $\pnpsi(K')$.
\end{proof}

\begin{lemma}\label{lem:q-gr-min}
  If $k\geq 0$, then over the vertex $u\times(0,0,\dots,0)$,
  $\Filt_{\SL(K)}\BNcx$ is all of $\BNcx$; that is, if $(v,x)\in\BNcx$
  with $v=u\times(0,0,0,\dots,0)$, then $\gr_q(v,x)\geq
  \SL(K)$. Similarly, if $k<0$, and if $(v,x)\in\BNcx$ with $v=u\times
  (\epsilon_1,\dots,\epsilon_{|k|+1})$ where exactly one of the
  $\epsilon_i$ is $0$, then $\gr_q(v,x)\geq\SL(K)$.
\end{lemma}
\begin{proof}
  The minimal quantum grading occurs when all of the circles are
  decorated by $x_-$. In each of the above cases, there are $m$
  circles in $K_v$ (here $(m+1)$ is the braid index). The weight $|v|$ is
  $n_--1$, so Formula~\eqref{eq:q-gr} gives
  \[
  \gr_q(v,(x_-,\dots,x_-))=n_+-2n_-+n_--1-m=n_+-n_--(m+1)=\SL(K),
  \]
  as claimed.
\end{proof}

\begin{proof}[Proof of Theorem~\ref{thm:flype-inv}]
  As noted earlier, the element $\phi$ is given by
  Corollary~\ref{cor:Jake-preserved}; we must show that
  $\phi\in\Filt_{\SL(K)}\BNcx(K')$.  But this follows from the third
  part of Lemma~\ref{lem:live-over} in conjunction with
  Lemma~\ref{lem:q-gr-min}.
\end{proof}

The story for the invariant $\dpsi$ is slightly subtler. The flype
isotopy from Formula~\eqref{eq:flype-moves} involves a single negative
stabilization followed by a transverse isotopy followed by a single
negative destabilization. Up to an overall sign, and relative
boundary, the negative stabilization sends $\ppsi-\npsi$ to
$\ppsi+\npsi$ (Remark~\ref{rem:neg-stabilization-linear}), the
transverse isotopy preserves $\ppsi+\npsi$
(Remark~\ref{rem:linear-combinations}), and the negative
destabilization sends $\ppsi+\npsi$ to $\ppsi-\npsi$. Therefore,
\[
f(\dpsi(K))=\pm\dpsi(K')+\diff\phi
\]
for some $\phi\in\BNcx^{-1}(K')$. Furthermore, as in the proof of
Theorem~\ref{thm:flype-inv}, the homological grading and locality of
$\phi$ force $\phi$ to lie in filtration level
$\Filt_{\SL(K)}\BNcx(K')$.  This shows that the invariant
$\dpsi_{0,\infty}$ is preserved under a general flype. However, for
special types of flypes, more can be said.

\begin{theorem}\label{thm:diff-flype-inv}
  With braids $A$ and $B$ as in Definition~\ref{def:flype}, consider
  the flype isotopy between the braid closures $K$ and $K'$ of
  $A\si_m^k B\si_m^{-1}$ and $A\si_m^{-1}B\si_m^k$. Furthermore,
  assume $k\geq 0$. If $f\from \BNcx(K)\to\BNcx(K')$ is the map
  associated to such a flype, then $f(\dpsi(K))=\pm\dpsi(K')+\diff
  \phi$ (for one choice of $+$ or $-$), where
  $\phi\in\Filt_{\SL(K)+2}\BNcx(K')$.
\end{theorem}

\begin{proof}
  From the discussion above, we see that
  $f(\dpsi(K))=\pm\dpsi(K')+\diff\phi$ for some
  $\phi\in\Filt_{\SL(K)}\BNcx^{-1}(K')$. Indeed, reusing earlier
  notations and proofs, we see that the homological grading and
  locality of $\phi$ forces it to lie over the vertex
  $v=u\times(0,0,\dots,0)$. Let $x$ be the labeling that labels each
  of the $m$ circles in $(K')_v$ by $x_-$. Then $\phi_0=(v,x)$ is the
  unique generator over $v$ that lies in quantum grading
  $\SL(K)$. Therefore, $\phi$ can be written as $\al\phi_0+\phi_1$ for
  some $\al\in\ZZ$ and some
  $\phi_1\in\Filt_{\SL(K)+2}\BNcx^{-1}(K')$.  It follows that
  \begin{equation}\label{eq:al}
  \al\diff(\phi_0)=f(\dpsi(K))\mp\dpsi(K')-\diff(\phi_1).
  \end{equation}
  If $\psi_{0,1}(K')$ denotes the original Plamenevskaya generator
  (which lies over the oriented resolution and labels all the circles
  by $x_-$), then $\langle\diff\phi_0,\psi_{0,1}\rangle=\pm 1$.
  However, all the terms of the right side of the Equation~\eqref{eq:al} lie
  in filtration level $\Filt_{\SL(K)+2}\BNcx(K')$. Since
  $\psi_{0,1}(K')$ is a homogenous element in quantum grading
  $\SL(K)$, this forces $\al=0$. Hence,
  $\phi=\phi_1\in\Filt_{\SL(K)+2}\BNcx^{-1}(K')$.
\end{proof}

\begin{remark}
We do not know if Theorem~\ref{thm:diff-flype-inv} holds generally
without the assumption $k\geq 0$. However,
for all of the examples we know where negative flypes produce
putatively distinct transverse knots, $k$ is at least $2$.
\end{remark}

\subsection{Triviality for simple links}\label{sec:sad}

In Section~\ref{sec:flypes}, we saw that the filtered Plamanevskaya
invariant $\pnpsi$ remains invariant under negative flypes. In this
subsection, we will see that $\pnpsi$ cannot distinguish
different transverse representatives of any link with particularly simple Khovanov
homology.

\begin{proposition}\label{prop:trivial_simple}
Let $K$ and $K'$ be transverse links with the same topological link
type and with $\SL(K)=\SL(K')$. Further assume
$H^{-1}(\BNcx(K)/\Filt_{\SL(K)}\BNcx(K))=0$. If
$f\from\BNcx(K')\to\BNcx(K)$ is the filtered chain map corresponding
to some sequence of Reidemeister moves connecting $K'$ to $K$, then
there exists some $\phi\in\Filt_{\SL(K)}\BNcx^{-1}(K)$ so that
\[
f(\pnpsi(K'))=\pm \pnpsi(K)+\diff\phi.
\]
\end{proposition}

\begin{proof}
Consider the following short exact sequence:
\[
\xymatrix{
0\ar[r]&\Filt_{\SL(K)}\BNcx^{-1}(K)\ar[r]^-{\iota}\ar[d]^-{\diff}&\BNcx^{-1}(K)\ar[r]^-{\pi}\ar[d]^-{\diff}&
\BNcx^{-1}(K)/\Filt_{\SL(K)}\BNcx^{-1}(K)\ar[r]\ar[d]^-{\diff}&0\\
0\ar[r]&\Filt_{\SL(K)}\BNcx^{0}(K)\ar[r]^-{\iota}&\BNcx^{0}(K)\ar[r]^-{\pi}&
\BNcx^{0}(K)/\Filt_{\SL(K)}\BNcx^{0}(K)\ar[r]&0.
}
\]
Since $f$ is a filtered map,
$f(\pnpsi(K'))\mp\pnpsi(K)\in\Filt_{\SL(K)}\BNcx^0(K)$. By
Corollary~\ref{cor:Jake-preserved}, there exists
$\theta\in\BNcx^{-1}(K)$, so that
\[
\diff\theta=\iota(f(\pnpsi(K'))\mp\pnpsi(K)).
\]
Therefore,
\[
\diff(\pi(\theta))=\pi(\diff\theta)=\pi(\iota(f(\pnpsi(K'))\mp\pnpsi(K)))=0,
\]
and hence $[\pi(\theta)]\in H^{-1}(\BNcx(K)/\Filt_{\SL(K)}\BNcx(K))$
represents some homology element. Furthermore, the connecting
homomorphism $H^{-1}(\BNcx(K)/\Filt_{\SL(K)}\BNcx(K))\to
H^0(\Filt_{\SL(K)}\BNcx(K))$ maps $[\pi(\theta)]$ to
$[f(\pnpsi(K'))\mp\pnpsi(K)]$. Since
$H^{-1}(\BNcx(K)/\Filt_{\SL(K)}\BNcx(K))=0$, $[\pi(\theta)]=0$ and hence
$[f(\pnpsi(K'))\mp\pnpsi(K)]=0$ in $H^0(\Filt_{\SL(K)}\BNcx(K))$.
Therefore, there exists some $\phi\in\Filt_{\SL(K)}\BNcx^{-1}(K)$ so
that
\[
f(\pnpsi(K'))\mp\pnpsi(K)=\diff\phi.\qedhere
\]
\end{proof}

\begin{corollary}\label{cor:filtered_trivial}
  Let $K$ be a topological link, and let $\maxsl(K)$ be the
  maximal self-linking number among all transverse representatives of
  $K$. If $Kh^{-1,j}(K;\Z)=0$ for all $j<\maxsl(K)$, then the
  filtered Plamenevskaya invariant does not distinguish transverse
  representatives of $K$ with the same self-linking number.
\end{corollary}

\begin{proof}
  This follows immediately from
  Proposition~\ref{prop:trivial_simple}. Since the Khovanov chain
  complex is the associated graded object of the Bar-Natan chain
  complex, $\bigoplus_{j<q}\Kh^{-1,j}(K)=0$ implies
  $H^{-1}(\BNcx(K)/\Filt_q\BNcx(K))=0$.
\end{proof}

The maximal self-linking number is known for all knots through $11$
crossings \cite{Ng-arc}, and Corollary~\ref{cor:filtered_trivial}
holds for all topological knot types up to $11$ crossings except the ones in
the following list (with $\ol{K}$ denoting the topological mirror of
$K$):

\parbox{0.8\textwidth}{
$\ol{8_{20}}$,
$\ol{10_{125}}$,
$\ol{10_{126}}$,
$\ol{10_{130}}$,
$\ol{10_{141}}$,
$\ol{10_{143}}$,
$\ol{10_{148}}$,
$10_{155}$,
$\ol{10_{159}}$,
$11n_{22}$,
$11n_{26}$,
$11n_{40}$,
$11n_{46}$,
$\ol{11n_{50}}$,
$11n_{51}$,
$11n_{54}$,
$\ol{11n_{65}}$,
$11n_{71}$,
$\ol{11n_{75}}$,
$\ol{11n_{87}}$,
$\ol{11n_{127}}$,
$\ol{11n_{132}}$,
$\ol{11n_{138}}$,
$11n_{146}$,
$\ol{11n_{159}}$,
$\ol{11n_{172}}$,
$\ol{11n_{176}}$,
$11n_{178}$,
$11n_{184}$.}

\noindent We currently do not know of any distinct transverse representatives
with the same self-linking number, flype-equivalent or otherwise, of
any of the above knot types.

\section{A cohomotopy refinement of the graded Plamenevskaya invariant}\label{sec:cohomotopy}
In~\cite{RS-khovanov} a space-level refinement of Khovanov homology
was given. That is, let $K$ be a link diagram. For each $j$ there is a
(formal desuspension of a) suspension spectrum of a CW complex
$\KhSpace^j(K)$ so that
$\wt{H}^i(\KhSpace^j(K))\cong\Kh^{i,j}(K)$ and the homotopy type
of $\KhSpace^j(K)$ is a link invariant.

In this section we give a space-level refinement of the Plamenevskaya
invariant. That is:
\begin{theorem}\label{thm:cohtpy}
  Associated to a braid diagram $K$ is a map
  \[
  \Psi(K)\co \KhSpace^{\SL(K)}(K)\to \sphere,
  \]
  where $\sphere$ is the sphere spectrum.
  The induced map on cohomology
  \[
  \Psi(K)^*\co \ZZ=\wt{H}^0(\sphere)\to
  \wt{H}^0(\KhSpace^{\SL(K)}(K))\cong \Kh^{0,\SL(K)}(K)
  \]
  sends a generator of $\ZZ$ to the graded Plamenevskaya invariant
  $[\psi_{0,1}(K)]$.

  If $K'$ is another braid diagram representing the same transverse
  link type then there is a commutative diagram
  \[
  \xymatrix{
    \KhSpace^{\SL(K)}(K) \ar[r]^-{\Psi(K)}\ar[d]_{\Phi}^\simeq &
    \sphere\ar[d]^{\simeq}\\
    \KhSpace^{\SL(K)}(K') \ar[r]^-{\Psi(K')} &
    \sphere.
  }
  \]
  (Here, $\Phi$ is the homotopy equivalence induced by a sequence of
  transverse Markov moves connecting $K$ and $K'$, and the map $\sphere\to
  \sphere$ is a self-homotopy equivalence of the sphere spectrum.)

  In other words, there is a transverse invariant $\Psi(K)\in
  \pi^0_s(\KhSpace^{\SL(K)}(K))$, well-defined up to sign and automorphisms of
  $\KhSpace^{\SL(K)}(K)$.
\end{theorem}

\subsection{The definition of the invariant}
Recall that the space $\KhSpace^j(L)$ is defined by feeding a framed
flow category $\KhFlowCat^j(L)$ into the Cohen--Jones--Segal
machine~\cite{CJS-gauge-floerhomotopy} (see
also~\cite[Definition~\ref*{KhSp:def:flow-gives-space}]{RS-khovanov}). We
use $\Realize{\Cat}$ for the result of applying the Cohen--Jones--Segal
construction to a framed flow category; note that this is not the same
as taking the geometric realization of the category (even after
viewing it as a topological category).  We will need the following
properties of this construction:
\begin{enumerate}
\item\label{item:realize} The space $\Realize{\Cat}$ has one cell for
  each object of $\Cat$ (in addition to a basepoint).
\item\label{item:upwards} If $\Cat'$ is a full, upwards-closed
  sub-category of $\Cat$ then there is a quotient map
  $\Realize{\Cat}\to\Realize{\Cat'}$ (gotten by collapsing the cells
  corresponding to objects not in $\Cat'$ to the basepoint).
\item\label{item:objects} The category $\KhFlowCat^j(L)$ has one
  object for each generator $(v,x)$ of $\KhCx(L)$ with quantum grading
  $j$.
\item\label{item:Hom-empty} The space $\Hom((v,x),(w,y))$ is non-empty
  if and only if $(v,x)$ and $(w,y)$ can be connected by a sequence of
  differentials in $\KhCx(L)$.
\end{enumerate}

Consider the graded Plamenevskaya generator $(v,x)$. That is, $v$
corresponds to the oriented resolution and $x$ labels each circle in
$v$ by $x_-$.  By Property~(\ref{item:objects}), this generator
corresponds to an object $(v,x)$ of $\KhFlowCat^{\SL}(L)$.  The
generator $(v,x)$ is a cycle, and hence by
Property~(\ref{item:Hom-empty}) the object $(v,x)$ by itself is an
upwards-closed subcategory of $\KhFlowCat^{\SL}(L)$. By
Property~(\ref{item:realize}), the realization $\Realize{\{(v,x)\}}$ of
the subcategory $\{(v,x)\}$ is a sphere $S^N$.  By
Property~(\ref{item:upwards}), there is a quotient map
$\Realize{\KhFlowCat^{\SL}(L)}\to \Realize{\{(v,x)\}}=S^N$. Formally
desuspending $N$ times gives the \emph{cohomotopy Plamenevskaya
  invariant} $\Psi(K)\co \KhSpace^{\SL(K)}(K)\to \sphere$.

\subsection{Invariance under transverse isotopies}
We turn now to invariance of the cohomotopy Plamenevskaya invariant,
i.e., the proof of Theorem~\ref{thm:cohtpy}. First, one more piece of
notation: the map $\Psi(K)$ is induced by a particular cell
$e(K)\in\Realize{\KhFlowCat(K)}$, with the property that no
higher-dimensional cells are attached to $e(K)$. The map $\Psi(K)$ is
an umkehr (wrong way) map associated to $e(K)$, gotten by collapsing
$\Realize{\KhFlowCat(K)}\setminus \mathrm{interior}(e(K))$ to a point.

\begin{proof}[Proof of Theorem~\ref{thm:cohtpy}]
  We have already defined the map $\Phi$, and it is obvious from the
  construction that the induced map on cohomology is the graded
  Plamenevskaya invariant. It remains to prove invariance.  As in the
  proof of Theorem~\ref{thm:filt-inv}, we will check invariance under
  Reidemeister II, braid-like Reiemeiester III, and positive
  stabilization. But first we need to verify independence of the
  auxiliary data used to construct $\KhSpace(K)$.

  Fix a braid closure $K$. Recall that $\KhSpace(K)$ depends on the
  following auxiliary choices:
  \begin{itemize}
  \item An ordering of the crossings of $L$.
  \item A sign assignment $s$ for the cube $\Cube(n)$.
  \item A neat embedding $\iota$ and a framing $\Frame$ for the cube
    flow category $\CubeFlowCat(n)$ relative to $s$.
  \item A framed neat embedding of the Khovanov flow category
    $\KhFlowCat(L)$ relative to some $\TupV{d}$. This framed neat
    embedding is a perturbation of $(\iota,\Frame)$.
  \item Integers $A,B$ and real numbers $\ep,R$ used in the
    construction of the CW complex.
  \end{itemize}
  It is shown
  in~\cite[Proposition~\ref*{KhSp:prop:choice-independent}]{RS-khovanov}
  that up to homotopy equivalence, the space $\KhSpace(K)$ is independent of
  these auxiliary choices. The maps coming from changing $\epsilon$
  and $R$ take each cell homeomorphically to the corresponding cell,
  and hence commute with the maps $e$ and, hence, $\Phi$
  (see~\cite[Lemma~\ref*{KhSp:lem:CW-indep-ep-R-framing}]{RS-khovanov}). The
  maps for changing $A$, $B$ and $\TupV{d}$ again take each cell by a
  degree $1$ map to the corresponding cell (see
  ~\cite[Lemma~\ref*{KhSp:lem:CW-indep-A-B}]{RS-khovanov}) so again
  commute with $e$ and so $\Phi$. Any two perturbations of $\iota$ can
  be connected by a $1$-parameter family. The corresponding map again
  is a homeomorphism taking cells to corresponding cells (this again
  comes from
  Lemma~\cite[Lemma~\ref*{KhSp:lem:CW-indep-ep-R-framing}]{RS-khovanov}),
  so again commutes with $e$ and $\Phi$. The same applies to changing
  $\iota$.

  The proof of independence of the sign assignment goes as
  follows. Let $K'$ be the disjoint union of $K$ with a $1$-crossing
  unknot $U$, so that the $0$-resolution of $U$ has two components
  (say). The cube of resolutions of $K'$ has the cube of resolutions
  of $K$ as two faces, say $f_0$ and $f_1$, corresponding to taking
  the $0$ and $1$ resolution at $U$, respectively. Given sign assignments
  $s_0$ and $s_1$ for $K$ there is a sign assignment $s$ for $K'$ such
  that $s|_{f_i}=s_i$. Consider the subcomplex $Y$ of
  $\Realize{\KhFlowCat(K';s)}$ in which the circle(s) corresponding to
  $U$ are labeled by $x_+$. The subcomplex $Y$ is contractible
  (because $\wt{H}^*(Y)=0$), and there is a cofibration sequence
  \[
  \Realize{\KhFlowCat(K;s_0)}\to Y \to \Realize{\KhFlowCat(K;s_1)}.
  \]
  The Puppe construction then gives the desired homotopy equivalence.
  Now, there is a map $\Xi\co Y\to \DD^{N+1}$ (for an appropriate $N$)
  so that the following diagram commutes:
  \[
  \xymatrix{
    \Realize{\KhFlowCat(K;s_0)}\ar[r]\ar[d]_{\Psi(K;s_0)}&
    Y\ar[r]\ar[d]_{\Xi} & \Realize{\KhFlowCat(K;s_1)}\ar[d]^{\Psi(K;s_1)}\\
    S^N \ar[r]& S^N\cup\DD^{N+1}=\DD^{N+1}\ar[r] & S^{N+1}=\DD^{N+1}/\bdy\DD^{N+1}.
  }
  \]
  Specifically, the map $\Xi$ is given by sending the cell in $f_0$
  corresponding to $e(K)$ to $S^{N}$ and the cell in $f_1$
  corresponding to $e(K)$ to $\DD^{N+1}$. It follows that the Puppe
  map
  $\Realize{\KhFlowCat(K;s_1)}\to\Sigma\Realize{\KhFlowCat(K;s_0)}$
  commutes with $\Psi$.

  Finally, changing the ordering of the crossings has the same effect as a
  particular change of sign assignment, giving invariance under this
  as well. This completes the proof of invariance under the auxiliary
  choices.

  Next, we turn to Reidemeister invariance.  It is convenient in all
  cases to use the ``sub-complex of quotient complex'' trick that we
  used in Theorem~\ref{thm:filt-inv} (where the complexes are, in fact, drawn
  from~Sections~\ref{subsec:r1p}--\ref{subsec:r3}).

  We start with Reidemeister II. Suppose $K'$ is obtained from $K$ by
  a Reidemeister II move introducing two new crossings. Let $n$ be the
  number of crossings in
  $K$. In~\cite[Proposition~\ref*{KhSp:prop:RII}]{RS-khovanov}, the first and third authors
  constructed an upwards-closed subcategory $\Cat_1$ of
  $\KhFlowCat(K')$ so that $\Realize{\Cat_1}$ is contractible. Let
  $\Cat_2$ denote the complementary, downwards-closed subcategory of
  $\KhFlowCat(K')$. A further downwards-closed
  subcategory $\Cat_3$ of $\Cat_2$ was constructed so that $\Realize{\Cat_3}$ is
  contractible and so that the complement $\Cat_4$ of $\Cat_3$ is
  exactly $\KhFlowCat(K)$. Indeed, there is a vertex $u\in\{0,1\}^2$
  so that resolving the two new crossings of $K'$ according to $u$
  gives $K$; and $\Cat_4$ is the sub-cube $\{0,1\}^n\times
  \{u\}\subset \{0,1\}^{n+2}$. In particular, after making compatible
  choices of framed embeddings for the flow categories, there is a map
  $\mathbb{D}^n\to \Realize{\Cat_2}$ making the following
  diagram commute (on the nose):
  \[
  \xymatrix{
    & \mathbb{D}^N \ar[dl]_{e(K')}\ar[dr]^{e(K)}\ar[d]& \\
    \Realize{\KhFlowCat(K')} & \Realize{\Cat_2}\ar[l]^-{\hookleftarrow\Realize{\Cat_2}}_-\simeq\ar[r]_-{/\Realize{\Cat_3}}^-\simeq & \Realize{\KhFlowCat(K)}.
  }
  \]
  Reidemeister II invariance of $\Psi$ follows by replacing $e(K)$ and
  $e(K')$ by the umkehr maps $\Psi(K)\co \Realize{\KhFlowCat(K)}\to
  \mathbb{D}^N/\bdy\mathbb{D}^N$ and $\Psi(K')\co
  \Realize{\KhFlowCat(K')}\to \mathbb{D}^N/\bdy\mathbb{D}^N$.

  The proof of braid-like Reidemeister III invariance is essentially
  the same; only the definitions of the $\Cat_i$ change
  (see~\cite[Proposition~\ref*{KhSp:prop:RIII}]{RS-khovanov}).

  Finally, stabilization invariance is slightly easier. Suppose that
  $K'$ is obtained from $K$ by a positive stabilization. The proof
  of~\cite[Proposition~\ref*{KhSp:prop:RI}]{RS-khovanov} gives an
  upwards-closed subcategory $\Cat_1$ of $\KhFlowCat(K')$ so that
  $\Realize{\Cat_1}$ is contractible, and the complementary
  downwards-closed subcategory $\Cat_2$ of $\KhFlowCat(K')$ is
  identified with $\KhFlowCat(K)$. It is immediate from the definition
  of $\Cat_1$ that the image of $e(K')\co \mathbb{D}^N\to
  \Realize{\KhFlowCat(K')}$ lies inside $\Realize{\Cat_2}$; and agrees
  with $e(K)\co \mathbb{D}^N\to \Realize{\KhFlowCat(K)}$. So, again,
  $\Psi$ is invariant. This concludes the proof.
\end{proof}

\begin{remark}
  A careful reader will observe that in the construction of
  $\KhSpace(K)$, one also made a global choice of ladybug
  matching (see~\cite[Definition~5.6]{RS-khovanov}). For each of the two choices of ladybug matchings
  $\mf{m}_1$ and $\mf{m}_2$, one gets spectra $\KhSpace^j(K;\mf{m}_i)$
  and a transverse invariant
  $\Psi(K;\mf{m}_i)\in\pi^0_s(\KhSpace^{\SL(K)}(K;\mf{m}_i))$. Although
  we show
  in~\cite[Proposition~\ref*{KhSp:prop:ladybug-invariance}]{RS-khovanov}
  that the spectra $\KhSpace^j(K;\mf{m}_1)$ and
  $\KhSpace^j(K;\mf{m}_2)$ are (stably) homotopy equivalent, we do not
  know if these homotopy equivalences carry $\Psi(K;\mf{m}_1)$ to
  $\Psi(K;\mf{m}_2)$. So, whenever we talk about $\Psi(K)$ without referencing the
  choice of ladybug matching, it is implicit that we have already made
  some (global) choice of ladybug matching.

  On the other hand, let $T$ be a transverse link and
  $B=\sigma_{i_1}^{\epsilon_1}\sigma_{i_2}^{\epsilon_2}\cdots\sigma_{i_\ell}^{\epsilon_\ell}$
  a braid representing $T$. Then
  $B^*=\sigma_{i_\ell}^{\epsilon_\ell}\cdots\sigma_{i_2}^{\epsilon_2}\sigma_{i_1}^{\epsilon_1}$
  represents a potentially different transverse link $T^*$, called the
  \emph{transverse mirror} of $T$ \cite{NgThurston}.  The proof
  of~\cite[Proposition~\ref*{KhSp:prop:ladybug-invariance}]{RS-khovanov}
  does show that there is a homotopy equivalence between
  $\KhSpace(K;\mf{m}_1)$ and $\KhSpace(K^*;\mf{m}_2)$ that carries
  $\Psi(K;\mf{m}_1)$ to $\Psi(K^*;\mf{m}_2)$.
\end{remark}

\subsection{Consequences and computable
  invariants}\label{sec:computable}

We conclude by stating some immediate properties of the cohomotopy
refinement of the Plemenevskaya invariant, and suggesting some
further (computable) auxiliary invariants.

\begin{corollary}
  The graded Plamenevskaya invariant $[\psi_{0,1}(K)]$ lies in the
  image of the co-Hurewicz map
  $\pi^0_s(\KhSpace^{\SL(K)}(K))\to \wt{H}^0_s(\KhSpace^{\SL(K)}(K))$.
  In particular, if $\pi^0_s(\KhSpace^{\SL(K)}(K))=0$ then
  $[\psi_{0,1}(K)]=0$.
\end{corollary}

\begin{corollary}\label{cor:cohtpy-boring} If the co-Hurewicz map
  $\pi^0_s(\KhSpace^{\SL(K)}(K))\to \wt{H}^0_s(\KhSpace^{\SL(K)}(K))$
  is injective then the cohomotopy Plamenevskaya invariant $\Psi(K)$
  is determined by the graded Plamenevskaya invariant
  $[\psi_{0,1}(K)]$. In particular, if $\Kh^{i,\SL(K)}(K;\Z)=0$ for all
  $i>0$ then $\Psi(K)$ is determined by $[\psi_{0,1}(K)]$.
\end{corollary}

\begin{proof}
  The first part is immediate since $\Psi(K)$ maps to
  $[\psi_{0,1}(K)]$ under the co-Hurewicz map. The second part follows
  from the first part and the Hopf classification theorem (which is
  dual to Hurewicz theorem; see~\cite{Spa-top-duality}) which asserts
  that if $\Kh^{i,\SL(K)}(K;\Z)=\wt{H}^i(\KhSpace^{\SL(K)}(K);\Z)=0$ for all
  $i>0$, then the co-Hurewicz map $\pi^0_s(\KhSpace^{\SL(K)}(K))\to
  \wt{H}^0_s(\KhSpace^{\SL(K)}(K))$ is an isomorphism.
\end{proof}
\noindent It can be checked that Corollary~\ref{cor:cohtpy-boring} applies to
all topological knot types up to $11$ crossings, and all $12$-crossing
knots except for $12n_{749}$. That is, for these knots,
$\Kh^{i,j}(K;\Z)=0$ for all $i>0$ and all
$j\leq\maxsl(K)$ (recall $\maxsl$ denotes the maximal self-linking
number).
 Therefore, for any transverse representative
of any of these knot types, $\Psi$ will be determined by
$[\psi_{0,1}]$. To date, we have not been able to use the cohomotopy
Plamenevskaya invariant $\Psi$ to distinguish transverse links with
the same self-linking number.

We conclude by mentioning two strategies for using the
cohomotopy invariant or the Khovanov homotopy type that could
conceivably be useful.
Suppose that $K_1$ and $K_2$ are transverse representatives for the
same topological link type.

\begin{idea}\label{idea:sq-image}The spectrum $\KhSpace(K)$ induces
  stable cohomology operations like the Steenrod squares $\Sq^k\co
  \wt{H}^i(\KhSpace^j(K);\Field_2)\to\allowbreak
  \wt{H}^{i+k}(\KhSpace^j(K);\Field_2)$.  Perhaps for some choice of
  $k$, we have $[\psi_{0,1}(K_1)]\in\image(\Sq^k)$ and
  $[\psi_{0,1}(K_2)]\not\in\image(\Sq^k)$. Or, perhaps for some choice
  of $k$, $[\psi_{0,1}(K_1)]\in\ker(\Sq^k)$ and
  $[\psi_{0,1}(K_2)]\not\in\ker(\Sq^k)$.  Since the Reidemeister
  isomorphisms commute with the action of $\Sq^k$, either of these
  phenomena would distinguish $K_1$ and $K_2$.
\end{idea}
\noindent In general, computing the Steenrod squares for an arbitrary CW complex
(with exponentially many cells) is not so easy. In~\cite{RS-steenrod},
it is explained how to compute the operation $\Sq^2$, and this would be a
reasonable place to start exploring Idea~\ref{idea:sq-image}.

\begin{idea}\label{idea:Hopf}
  Let $\Cone(\Psi)$ denote the mapping cone of $\Psi$. Then the stable
  homotopy type of $\Cone(\Psi(K))$ is a transverse invariant. So,
  perhaps for some choices of $K_1$ and $K_2$, the mapping cones of
  $\Psi(K_i)$ distinguish the $K_i$.
\end{idea}
\noindent For example, one could try to use the action of Steenrod
squares to distinguish $\Cone(\Psi(K_1))$ and $\Cone(\Psi(K_2))$. In
particular, it should be reasonably straightforward to extend
techniques from~\cite{RS-steenrod} to compute the operation $\Sq^2$ on
$\Cone(\Psi(K))$.

\bibliographystyle{myalpha}
\bibliography{newbibfile}

\end{document}